\documentclass[11pt, reqno]{article}
\usepackage{amssymb}
\usepackage{latexsym}
\usepackage{amsmath}
\usepackage{bbm}
\usepackage{amsthm}
\usepackage{stmaryrd}
\usepackage{amsfonts}
\usepackage{mathtools}
\usepackage{color}
\usepackage{hyperref}
\usepackage{pdfsync}
\usepackage{enumitem}
\usepackage[usenames,dvipsnames]{pstricks}
\usepackage{epsfig}
\usepackage{pst-grad} 
\usepackage{pst-plot} 
\usepackage{xcolor} 
\usepackage{sectsty}
\usepackage{cleveref}
\usepackage[mathscr]{eucal}

\sectionfont{\normalfont\centering}
\subsectionfont{\fontsize{12}{15}\selectfont}

\usepackage[scr=boondox]{mathalpha}

\usepackage[T1]{fontenc}
\usepackage{fbb}

\numberwithin{equation}{section}
\newtheorem{theorem}{Theorem}[section]
\newtheorem{definition}[theorem]{Definition}
\newtheorem{lemma}[theorem]{Lemma}
\newtheorem{assumption}[theorem]{Assumption}
\newtheorem{proposition}[theorem]{Proposition}
\newtheorem{corollary}[theorem]{Corollary}
\newtheorem{example}[theorem]{Example}
\newtheorem{remark}[theorem]{Remark}

\newcommand{\R}{\mathbb{R}}
\newcommand{\N}{\mathbb{N}}
\newcommand{\LSC}{\mathrm{LSC}}

\newcommand{\kk}{\mathscr{k}}

\date{\today}
\begin{document}

\begin{center}
{\LARGE On (discounted) global Eikonal equations in metric spaces}

\bigskip

\textsc{Tr\'i Minh L\^e \& Sebasti\'an Tapia-Garc\'ia}

\bigskip \textsc{\today}
\end{center}

\bigskip

\noindent\textbf{Abstract.} Eikonal equations in metric spaces have strong connections with the local slope operator (or the De Giorgi slope). 
In this manuscript, we explore and delve into an analogous model based on the global slope operator, expressed as $\lambda u + G[u] = \ell$, where $\lambda \geq 0$. 
In strong contrast with the classical theory, the global slope operator relies neither on the local properties of the functions nor on the structure of the space, and therefore new insights are developed in order to analyze the above equation.
Under mild assumptions on the metric space $X$ and the given data $\ell$, we primarily discuss: $(a)$ the existence and uniqueness of (pointwise) solutions; $(b)$ a viscosity perspective and the employment of Perron's method to consider the maximal solution; $(c)$ stability of the maximal solution with respect to both, the data $\ell$ and the discount factor $\lambda$.
Our techniques provide a method to approximate solutions of Eikonal equations in metric spaces and a new integration formula based on the global slope of the given function. 
\bigskip


\vspace{0.6cm}

\textit{Secondary} 35F30, 49L25.


\tableofcontents

\section{Introduction}
Let $\Omega \subset \R^n$ be a nonempty open bounded set and let $\ell: \Omega \to (0, + \infty)$  be a continuous function. The Eikonal equation reads as follows
    \begin{align}\tag{$\EuScript{E}$}\label{eqn.eikonal}
        \begin{dcases}
            \| Du(x) \| = \ell(x),& \text{for all } x \in \Omega, \\
            u = 0, &\text{for all }x \in \partial \Omega.
        \end{dcases}
    \end{align}
This equation has been extensively studied in several settings (Hilbert and Banach spaces~\cite{CL_1985, CL_1986}) because of its many applications.
It is an important model of the larger family of the so-called Hamilton--Jacobi equations (HJEs for short), all of them having the form $H(x, u, Du)=0$. 
The notion of viscosity solutions has been developed to remedy the possible lack of classical solutions. It became essential in the study of HJEs and, in particular of the Eikonal equations.

\medskip

During the last decades, there has been a growing effort to generalize equation~\eqref{eqn.eikonal} to a purely metric setting motivated by applications in  Optimal Transportation, Mean Field Games and Optimal Control on Topological Networks~\cite{ACCT_2013, AGS_2013, CCM_2016, C_2013, IMZ_2013, SC_2013}. 
In the literature, we can find several notions of solutions of HJEs defined on metric spaces. Some of them are based on the notion of viscosity solution. 
Most notably, according to \cite{LSZ_2021}, slopes-- and curve--based solutions are defined and studied in~\cite{AF_2014, GS_2015, GHN_2015}.
In~\cite{LSZ_2021}, assuming that the underlying metric space is complete and length, and under some extra continuity conditions on the data $\ell$, it is shown that these two notions coincide with Monge solutions. 
Moreover, Monge solutions of metric HJEs with discontinuous Hamiltonians have been addressed in~\cite{BD_2005, EGV_2024, LSZ_2024,NS_1995}.
This notion of solution does not require any test functions.
Indeed, Monge solutions for the Eikonal equation in metric spaces are pointwise solutions of what we call the local slope equation, based on the local slope operator (or De Giorgi slope). 
That is, for a metric space $X$ and a continuous function  $\ell: X \to [0, +\infty)$, we search for a function $u:X\to \R \cup \{ + \infty \}$ that pointwise satisfies
    \begin{equation}\label{eq.loca Eiko}
        \begin{dcases}
            s[u](x) = \ell(x), & \text{ for all } x \in X, \\
            u = 0, & \text{ for all } x \in [\ell = 0],
        \end{dcases}
    \end{equation}
where the local slope operator is defined by
    \begin{align*}
        s[u](x) := 
        \begin{dcases}
            \limsup_{y \to x} \dfrac{(u(x) - u(y))_+}{d(x, y)}, & \text{ if $x$ is not isolated, } \\
            \phantom{arisdaniilidis}0, & \text{ otherwise,}
        \end{dcases}
    \end{align*}
where $\alpha_+ = \max \{ \alpha, 0 \}$ for any $\alpha \in \R$. The convention $s[u](x)=+\infty$ if $u(x)=+\infty$ is used. 

\medskip

The local slope has been studied in depth as an upper gradient that generalizes the norm of the derivative of a differentiable function~\cite{AGS_2013}. 
Another natural upper gradient is the so--called global slope. That is, 
    \begin{align*}
        G[u](x) = \begin{cases}
            \displaystyle\sup_{y \neq x} \dfrac{(u(x) - u(y))_+}{d(x, y)},&\text{if } u(x)< + \infty,\\
            +\infty,& \text{otherwise}.
        \end{cases}
    \end{align*}  
Clearly, $G[u]\geq s[u]$ for any function $u$, and $G[u]=s[u]$ when $u$ is convex. 
Further information on upper gradients can be found in~\cite{AGS_2013} and references therein. 

\medskip

The main purpose of this manuscript is to study the (discounted) global Eikonal equation, see equation~\eqref{eqn: Gs} in Section~\ref{sec: global slope}.
Let us focus on the global Eikonal equation without discount term: 
	\begin{equation}\tag{$\EuScript{G}_{0}$}\label{eqn: G0}
    	\begin{dcases}
    		G[u](x)=\ell(x),&~\text{for all }x\in X,\\
    		\inf_{x\in X} u(x) = 0,&
    	\end{dcases}
	\end{equation}
where $\ell: X \to [0, + \infty)$ is a lower semicontinuous (lsc for short) function such that $\inf_X \ell = 0$. 
A lsc function $u: X \to [0, + \infty)$ with $\inf_X u =0$ is called a solution of~\eqref{eqn: G0} if $G[u]=\ell$ holds pointwise. 
Note that the above equation can be a model for phenomena in which the underlying metric space has no length structure, for instance, in graphs or in fractal domains. In the Euclidean setting, an equation involving a global slope-like operator appeared in the study of H\"older infinity Laplacians \cite{CLM_2013}. 
\medskip

Let us discuss about the setting of equation~\eqref{eqn: G0}.
The assumption of lower semicontinuity of $\ell$ follows from the fact that the global slope of a lsc function is lsc, Propostion~\ref{prop: 22}. 
A natural boundary condition for the Eikonal equation is the so-called Dirichlet condition.
However, this type of condition is not enough to ensure the uniqueness of solutions for global slope equations. 
Indeed, observe that the functions $f, g, h : \R \to \R$ defined by $f(x) = x$, $g(x) = - x$ and $h(x) = \min \{ x, -x\}$ satisfy $G[u] = 1$ in $\R$ and $u(0) = 0$.
Surprisingly, in \cite{DLS_2023, IZ_2023, TZ_2023}, it is shown that the equation \eqref{eqn: G0} admits at most one (bounded from below) solution. 
This explains that our boundary condition reads as $\inf_X u=0$ (i.e. we search for functions that are bounded from below).
Also, a simple consequence of Ekeland's principle shows that the infimum of the global slope of a bounded from below function, which is defined on a complete metric space, is $0$. 
Therefore, the assumption $\inf_X \ell = 0$ is a necessary condition for the existence of solutions when the underlying space is complete.

\medskip

To construct a solution of~\eqref{eqn: G0}, we study the asymptotic behavior of the semigroup defined by the operator
	\begin{equation}\label{eqn.LO}
		T_{\ell}\,u(x) := \inf \, \left\{ u(y) + \ell(x) d(x, y): y \in X \right\}, \text{ for all } x \in X. 
	\end{equation}

\medskip

In the language of discrete weak KAM theory  \cite{Z_2010, Z_2012}, the operator $T_\ell$ corresponds to a (degenerated) discrete Lax--Oleinik semigroup associated with the equation~\eqref{eqn: G0}. 
In the mentioned theory, the term $\ell(x)d(x, y)$ in the definition of $T_\ell$ is replaced by a \textit{continuous} cost function $c: X \times X \to \R$. Moreover, suitable assumptions are made on the space $X$. 
This theory has been developed as a discrete model of the weak KAM theory~\cite{F_2008}. 
We would like to emphasize that, for existence of solutions of our equation, no assumptions on the space $X$ are made. 

\medskip

An important portion of this work is devoted to analyzing the existence of solutions, the uniqueness issue, and the stability of equation~\eqref{eqn: Gs}, which reads as follows: $\lambda u + G[u] = \ell$. 
Concerning the existence of viscosity solutions of HJEs in the classical setting, there are two main approaches to build a solution: constructing the value function of a related control problem~\cite{BC-D_97} and the Perron's method~\cite{I_1987_Perron}. 
Note that usually, Perron's method does not give a formula to compute the provided solution.  
In the metric setting, for the local slope equation, several authors have considered a solution that is obtained through the value function of a control problem among absolutely continuous curves, addressing as well the uniqueness issue for their respective equations~\cite{AF_2014, GS_2015, GHN_2015}. 
In the language of determination theorems, uniqueness results have been tackled for slope--like operators, with the extra assumption that the functions are bounded from below.
Indeed, in \cite{BCD_2018} it is shown that two $\mathcal{C}^2$--smooth convex functions (which are bounded from below) defined on a Hilbert space coincide up to an additive constant if and only if their local slopes coincide (if and only if their global slopes coincide). This determination result was generalized threefold: from smooth convex functions to merely lsc functions, from Hilbert spaces to complete metric spaces, and from local slope to general descent modulus \cite{DLS_2023, DMS_2022, IZ_2023, PSV_2021, TZ_2023}. 
As a direct consequence of \cite[Theorem 1]{IZ_2023}, if the underlying space is complete, there is at most one solution of~\eqref{eqn: G0}.

\medskip

Coming back to the classical framework of HJEs, some classical stability results are provided in~\cite{CL_1983, CL_1986}. 
In the metric framework, we can also find stability results for metric HJEs in~\cite{GHN_2015, NN_2018}. 
In the convex case, roughly speaking, it has been shown that the difference between two functions on a bounded set is controlled by the $\mathcal{L}^\infty$-norm of the difference of their local slopes on that set (equivalently of their global slopes)~\cite{DD_2023}.
Further, for lsc convex functions defined on finite-dimensional spaces, the stability under the $\Gamma$--convergence of the slope is established in~\cite{DST-G_2024}.  

\medskip

A third approach to solve HJEs of the form $H(x, Du) = 0$ is the so--called ergodic approximation method, which was introduced in \cite[Section 2]{LPV_1987} to study the cell problem that arises in the periodic homogenization of HJEs. 
This method considers an auxiliary equation involving a discount factor: $\lambda u + H(x, Du) = 0$, where $\lambda > 0$, and then studies the asymptotic behavior of the solutions as $\lambda$ tends to $0^+$. 
In the framework of weak KAM theory, the ergodic approximation method was employed in~\cite{DFIZ_2016_invent} where the underlying space is a smooth compact connected manifold.  In the compact metric setting, a notion of discounted Lax--Oleinik semigroup was introduced in~\cite{DFIZ_2016}, in which the discount factor $\lambda$ is plugged directly into the definition of the discrete Lax--Oleinik semigroup. 
Indeed, in the aforementioned work, the authors considered the operator 
    \begin{align*}
        \mathcal{T}_{\lambda}u(x) = \inf_{y \in X} \,  \lambda u(y) + c(x, y), \text{ for all } x \in X,
    \end{align*}
and proved that the $\lambda$--discrete weak KAM solutions $\{ u_{\lambda} \}_{\lambda \in (0, 1)}$ (that is $\mathcal{T}_{\lambda}u_\lambda=u_\lambda+\beta$, for some $\beta\in \R$) converges to $u_1$ as $\lambda$ tends to $1^-$.
The above research stream was recently investigated in the context of Shapley operators in \cite{CGMQ_2024}.
We would like to point out that, since we apply the ergodic approximation method to the global slope equation, our \textit{discounted operator} $T_\lambda$ does not have the shape of a discounted Lax--Oleinik semigroup (as it does for the case $\lambda=0$). 
Indeed, as it is defined in \eqref{oper-T}, the natural operator associated with equation~\eqref{eqn: Gs} is given by
\[ 
    T_{\lambda}u(x) = \inf_{y \in X} \dfrac{u(y) + \ell(x)d(x, y)}{1 + \lambda d(x, y)}, \text{for every } x \in X.
\]
Moreover, for us, fix points of $T_\lambda$ only characterize subsolutions of our equation, Proposition~\ref{prop: T subsolution}.
Therefore, new insights are required to study the discounted global slope equation~\eqref{eqn: Gs}. 

\medskip

Fix $(X, d)$ a metric space and $\ell: X \to [0, + \infty)$ a lsc function such that $\inf_X \ell = 0$. 
The main contributions of this paper are the following:
\begin{enumerate}
    \item \textit{Existence of solutions in metric spaces}: Based on a control approach, in \Cref{thm.exists.full} we show that, if $\lambda>0$, the equation~\eqref{eqn: Gs} always admits a solution. 
    For the case $\lambda=0$, we derive a necessary condition \eqref{eq: hyp}, see Assumption \ref{ass.}, to obtain the existence of solutions. 
    More precisely, for any $\lambda \geq 0$, we show that $T^\infty_\lambda u := \lim_{n \to \infty} T^n_\lambda u$ is a solution of \eqref{eqn: Gs}, if $u$ is a supersolution of \eqref{eqn: Gs}, see Definition~\ref{def.P-solu}.   
    We point out that \eqref{eq: hyp} is also a sufficient condition if $X$ is assumed to be complete.    
    See Section~3.1.
    \item \textit{Uniqueness results}: Assume that $X$ is complete. 
    As stated above, the results from~\cite{IZ_2023} provide the uniqueness of the solution for equation~\eqref{eqn: G0}.
    For the case in which $\lambda>0$, the uniqueness of (bounded Lipschitz) solutions of equation~\eqref{eqn: Gs} is proved if $\ell$ is bounded. 
    We establish a comparison principle by using a transfinite induction inspired by the one employed in \cite{DLS_2023,DMS_2022} for abstract (metric compatible) descent modulus. 
    See Section 3.3. 
    \item \textit{The viscosity approach and Perron's method}: 
    We provide a notion of viscosity solution of equation~\eqref{eqn: Gs} and prove that it is equivalent to the notion of pointwise solution.
    We employ Perron's method to define the maximal solution $u_\lambda$ of equation~\eqref{eqn: Gs}, which allows us to fix a particular solution even in cases where there are multiple solutions. 
    Moreover, an explicit formula for the maximal solution is provided. 
    Indeed, if $\lambda > 0$, then $u_{\lambda} = T_\lambda^\infty \frac{\ell}{\lambda}$. On the other hand, if $\lambda = 0$ and assumption \eqref{eq: hyp} is satisfied, then 
        \begin{equation}\label{u_ley}
        u_0(x):=\inf\left\{\sum_{n=0}^\infty \ell(x_n)d(x_n,x_{n+1}):~\{x_n\}_n \in \mathcal{K}(x)\right\}, \text{ for every } x \in X,
        \end{equation}  
    where $\mathcal{K}(x) = \{ \{x_n\}_n \subset X: x_0=x,~\lim_{n \to \infty} \ell(x_n)=0 \}$. See \Cref{sec: viscosity}. 
    \item \textit{Stability result and ergodic approximation}: Assume that $X$ and $\ell$ are bounded. 
    We show that, for every $\lambda \geq 0$, the maximal solution of equation~\eqref{eqn: Gs} is $\mathcal{L}^\infty$--stable under uniform perturbation of the data $\ell$.
    Our method is based on proving a uniform bound of the length of the initial part of a discrete path (one may think on a partial sum of the series that appears in the formula~\eqref{u_ley}). 
    Intuitively, this bound reflects the finite length of gradient flows restricted to a set of large gradients.
    Moreover, using the $\mathcal{L}^\infty$--stability result for equation \eqref{eqn: G0}, we prove that
    \begin{itemize}[label = {$\bullet$}]
        \item $u_\lambda$ converges to $u_0$ uniformly as $\lambda$ tends to $0^+$;
        \item if $X$ is compact, then $u_\lambda$ converges to $u_\alpha$ uniformly as $\lambda$ tends to $\alpha$, for any $\alpha > 0$.
    \end{itemize}
    Furthermore, for any metric space, any lsc function $\ell$ and any $\alpha>0$, $u_{\lambda}$ converges to $u_{\alpha}$ pointwise as $\lambda$ tends to $\alpha^-$. See Section \ref{sec: stability}. 

    \item \textit{Approximations of a solution of the local slope equation~\eqref{eq.loca Eiko}}: 
    Assume that $[\ell = 0] \neq \emptyset$.
    Motivated by the explicit formula of the maximal solution of equation~\eqref{eqn: G0}, see~\eqref{u_ley}, we define a family of functions $\{ v_R \}_{R \in (0, + \infty)}$ that enjoys the following properties:
        \begin{itemize}[label = {$\bullet$}]
            \item if $X$ is connected, then $v_R$ converges pointwise to the maximal solution of~\eqref{eqn: G0} as $R$ tends to $+ \infty$;
            \item if $X$ is compact and length and $\ell$ is continuous, then $v_R$ converges uniformly to a solution of~\eqref{eq.loca Eiko} as $R$ tends to $0$. 
        \end{itemize}
    For $R > 0$, $v_R(x)$ is defined by the formula \eqref{u_ley} but replacing $\mathcal{K}(x)$ by 
        \begin{equation}
            \mathcal{K}_R(x) := \left\{ \{ x_n \}_{n = 0}^k \subset X: x_0 = x, x_{n + 1} \in B_R(x_n) \text{ for all } n < k \text{ and } x_k \in [\ell = 0] \right\}. 
        \end{equation}
    We also present an example to show that the second approximation result does not hold in general complete length spaces. 
    We would like to point out that a similar approximation scheme was already addressed in \cite{E-B_2023} but considering minimal weak upper gradients on metric measure spaces.
    See \Cref{subsec: local-global}. 
    \item \textit{Integration formula}: Assume that $X$ is complete. After extending our existence result for equation~\eqref{eqn: G0} to the case that $\ell$ admits the value $+ \infty$, we establish an integration formula that recovers lsc functions using only their global slopes. 
    More precisely, for a lsc function $f: X \to \R \cup \{ + \infty \}$ such that $\inf_X f = 0$, we have that
        \begin{align*}
            f = \mathrm{lsc}(\mathcal{I}[f]),
        \end{align*}
    where $\mathrm{lsc}(\mathcal{I}[f])$  is the lower semicontinuous envelope of the function $\mathcal{I}[f]$, which is defined by
        \begin{align*}
 	  \mathcal{I}[f](x) :=  \inf \left\{ \sum_{n = 0}^{+ \infty} G[f](x_n)d(x_n, x_{n + 1}): x_0 = x \text{ and } \lim_{n\to+\infty} G[f](x_n) = 0 \right\}.
 	\end{align*}
    This integration formula could be used on lsc convex functions to obtain a representation that differs from the one given by Rockafellar's integral. See Section~\ref{subsec: integration}.  
\end{enumerate}
\section{Preliminaries}\label{sec: preliminaries}
Let $ (X, d)$ be a metric space (which may be complete or not). 
For any $x \in X$ and $\rho > 0$, we denote by $B_{\rho}(x)$ the open ball centered at $x$ with radius $\rho$. 
Given a function $u: X \rightarrow \R \cup \{ + \infty \}$ and $r \in \R$, the $r$--sublevel set (resp. $r$--level set) is denoted by    
    \begin{align*}
        [u \leq r] := \{ x \in X: u(x) \leq r \} \hspace{0.5cm}\text{ (resp. $[u = r] := \{ x \in X : u(x) = r \}$)}.
    \end{align*}
The (effective) domain of $u$ is denoted by $\mathrm{dom}~u = \{ x \in X: u(x) < + \infty \}$ and $u$ is called proper if $\mathrm{dom}~u \neq \emptyset$. We also consider the space of real-valued lsc functions on $X$:
    \begin{align*}
        & \LSC(X) = \left\{ u \in \R^X : \quad  u \text{ is lsc} \right\}.
    \end{align*}
The following proposition collects some easy facts about the global slope. Its proof is left to the reader. 
\begin{proposition}\label{prop: 21}
    Let $u:X\to\R$ be a function and let $x\in X$. The following assertions hold true:
    \begin{itemize}
        \item If $G[u](x)<+\infty$, then $u$ is lsc at $x$.
        \item If $G[u]$ is everywhere finite, then $u\in \LSC(X)$.
        \item If $G[u]$ is locally bounded, then $u$ is locally Lipschitz.
    \end{itemize}
\end{proposition}
Note that the function $f:\R\to\R$ defined by $f(x)=x^3$ shows that the reciprocal of the above three implications are false in general. 
In the rest of the section, several facts of the global slope are provided.

\begin{proposition}\label{prop: 22}
    Let $u\in \LSC(X)$. 
    Then the function $G[u]:X\to\R\cup\{+\infty\}$ is lsc.
\end{proposition}

\begin{proof}
    If $u$ is constant, then $G[u]\equiv 0$ and there is nothing to prove. 
    Let us assume that $u$ is not constant and so, let $x\in X$ such $G[u](x)>0$ (possible $+\infty$). 
    Let $\alpha\in (0,G[u](x))$. 
    Then, by definition of $G[u]$, there is $y\in X\setminus\{x\}$ such that 
    \[\varphi_y(x):=\dfrac{u(x)-u(y)}{d(x,y)}>\alpha.\]
    Since $d$ is continuous and $u$ is lsc, $\varphi_y$ is lsc on $X\setminus \{y\}$.
    Let ${\delta\in (0, d(x,y))}$ be such that, for any $z\in B_{\delta}(x)$, $\varphi_y(z)>\alpha$. 
    So, for any $z\in B_{\delta}(x)$, we have that ${G[u](z)>\alpha}$.
    Since $\alpha$ can be chosen arbitrarily close to $G[u](x)$, $G[u]$ is lsc at $x$. 
    Finally, since $G[u]\geq 0$, we get that $G[u]$ is a lsc function.
\end{proof}

\begin{proposition}
    Let $X$ be a complete metric space and $u \in \LSC(X)$ be bounded from below. 
    Then, one has $\inf s_u(x)=\inf G[u]=0$.
\end{proposition}
\begin{proof}
    It follows from a direct consequence of Ekeland's variational principle.
\end{proof}
    
\begin{proposition}\label{prop.AC-seq}
    Let $X$ be a complete metric space.
    Let $u:X \rightarrow \R\cup\{+\infty\}$ be a bounded from below proper lsc function.
    Then, there is a sequence $\{z_n\}_{n}$ such that
    \[\lim_{n\to\infty}G[u](z_n)=0,~\text{and }\sum_{n=0}^{\infty} G[u](z_n)d(z_n,z_{n+1})<+\infty.\]
\end{proposition}

\begin{proof}
Let $\{ \epsilon_n \}_n\subset \R$ be a decreasing sequence convergent to $0$ such that $\epsilon_n/2 < \epsilon_{n + 1}$ for every $n \in \N$. 
Using Ekeland's variational principle (see, for instance~\cite[Theorem 1.4.1]{Z_2002}), there exists $z_0 \in X$ such that 
$G[u](z_0) \leq \epsilon_0 < + \infty$.
We construct a sequence $\{z_n\}_n$ by induction such that, for all $n\in \N$ we have 
\begin{align*}
    G[u](z_n)&\leq \epsilon _n,\\
    \dfrac{1}{2}G[u](z_n)d(z_n,z_{n+1})&\leq u(z_n)- u(z_{n+1}).
\end{align*}
Let $k\in \N$ and assume that we have already defined $\{z_n\}_{n=0}^k$ satisfying above inequalities.
Fix $\varrho \in (\epsilon_k/2, \epsilon_{k+1})$. 
Applying Ekeland's variational principle again, we find $z_{k+1} \in X$ such that
    \begin{equation}\label{eke.01}
        u(z_{k+1}) \leq u(z_k) - \varrho d(z_{k+1}, z_k)
    \end{equation}
and 
    \begin{equation}\label{eke.02}
        u(z_{k+1}) < u(y) + \varrho d(z_{k+1}, y), \, \text{ for all } y \neq z_{k+1}.
    \end{equation}
From \eqref{eke.02}, we get $G[u](z_{k+1}) \leq \varrho \leq \epsilon_{k+1}$. 
Combining \eqref{eke.01} and the fact that $G[u](z_k)/2 \leq \epsilon_k/2 < \varrho$, we obtain $\frac{1}{2}G[u](z_k) d(z_k, z_{k+1}) \leq u(z_k) - u(z_{k+1})$. 
The induction step is complete.

\medskip

Observe that $\lim_{n \rightarrow \infty} G[u](z_n) = 0$. 
Moreover, since $u$ is bounded from below, we have that
    \begin{align*}
        \sum_{n = 0}^{\infty} G[u](z_n)d(z_n, z_{n + 1}) \leq 2(u(z_1) - \inf_X u) < + \infty.
    \end{align*}
The proposition is proven.
\end{proof}

\begin{remark}\normalfont\label{Ekeland.rmk}
A careful inspection of the proof of Proposition~\ref{prop.AC-seq} gives us assertion $(i)$. Assertion $(ii)$ follows directly from Ekeland's principle.
	\begin{itemize}
		\item[$(i)$] For each $\sigma \in (0, 1)$ and $x \in \mathrm{dom}~G[u]$, there exists a sequence $\{ z_n \}_n$ starting at $x$ and depending on $\sigma$ such that 
    \begin{equation}
        \lim_{n \to \infty} G[u](z_n) = 0, \quad \sum_{n = 0}^{\infty} G[u](z_n)d(z_n, z_{n + 1}) < + \infty,
    \end{equation}
    and
    	\begin{equation}
    		(1 - \sigma) G[u](z_n)d(z_n, z_{n + 1}) \leq u(z_n) - u(z_{n + 1}), \quad \text{for all }n\in\N.
    	\end{equation}

    		\item[$(ii)$] For any $x \in \mathrm{dom} \, u$, there exists a sequence $\{ x_n \}_n \subset \mathrm{dom} \, G[u]$ such that 
                \[
               \lim_{n \to \infty} x_n = x \text{ and } \lim_{n \to \infty} u(x_n) \leq u(x).
                \] 
            Consequently, we always have $\mathrm{dom} \, u \subset \overline{\mathrm{dom}}\, G[u]$.
	\end{itemize}
\end{remark}

Lastly, for the sake of convenience, we recall a key lemma from~\cite{TZ_2023}, which will be used in the sequel.

\begin{proposition}\label{prop.TZ}{\normalfont \cite[Lemma 4.2]{TZ_2023}}
Let $X$ be a metric space and $u: X \to \R \cup \{ + \infty \}$ be a proper function. If a sequence $\{x_n\}_n \subset X$ satisfies
    \begin{align*}
        \lim_{n \to \infty} G[u](x_n) = 0 \quad \text{and} \quad \sum_{n = 0}^{ \infty}  G[u](x_{n + 1})d(x_n, x_{n + 1}) < + \infty,
    \end{align*}
then $\liminf_{n \to \infty} u(x_n) = \inf_X u$. 
\end{proposition}

\section{The (discounted) global slope equation}\label{sec: global slope}
Fix $(X,d)$ a metric space, $\lambda\geq 0$ and let $\ell:X\to\R$ be a lsc function such that $\inf_X \ell=0$. Consider the discounted version of equation~\eqref{eqn: G0}:
\begin{equation}\tag{$\EuScript{G}_{\lambda}$}\label{eqn: Gs}
    \begin{dcases}
    \lambda u(x)+G[u](x)=\ell(x),&~\text{for all }x\in X,\\
    {}\hspace{1.3cm}\inf_{x\in X} u(x) = 0.
    \end{dcases}
\end{equation}

In this section, we demonstrate that, under mild assumptions on the data $\ell$, the global slope equation~\eqref{eqn: Gs} always admits a solution. 
The uniqueness issue was already treated in \cite{IZ_2023} for the case when $X$ is complete and $\lambda=0$.  In \Cref{sec: uni}, we provide a uniqueness result when $X$ is complete, $\lambda>0$ and $\ell:X\to\R$ is bounded. 
Further, as a consequence of the above results, we characterize complete metric spaces in terms of the uniqueness of the mentioned equation, see Corollary~\ref{cor: characterization}.

\medskip

We are interested in pointwise solutions of equation~\eqref{eqn: Gs}. In the following definition we set the notions of sub- and supersolution that are used along this manuscript.
In Section~\ref{sec: viscosity} we characterize these notions from a viscosity perspective. 
Moreover, in \Cref{subsec: integration} we deal with the case in which $\ell$ takes the value $+\infty$.

\begin{definition}\label{def.P-solu}
A function $u: X \rightarrow \R\cup\{+\infty\}$ is said to be a subsolution of~\eqref{eqn: Gs} if $\textstyle\inf_X u = 0$ and 
    \begin{align*}
        \lambda u(x) + G[u](x) \leq \ell(x),\quad \text{for all }x\in X.
    \end{align*}
Similarly, a function $v: X\rightarrow \R\cup\{+\infty\}$ is said to be a supersolution of \eqref{eqn: Gs} if it is lsc, $\textstyle\inf_X v = 0$ and 
    \begin{align*}
        \lambda v(x) + G[v](x) \geq \ell(x),\quad \text{for all }x\in X.
    \end{align*}
A lsc function $u$ is a solution of~\eqref{eqn: Gs} if it is both sub-- and supersolution of ~\eqref{eqn: Gs}.
\end{definition}

\subsection{Existence results}\label{subsec: exist}

Thanks to Proposition~\ref{prop.AC-seq}, if $X$ is complete, the following hypothesis is a necessary condition for the existence of solutions of equation~\eqref{eqn: G0} (i.e. $\lambda=0$).

\begin{assumption}\label{ass.}
We say that the function $\ell:X\to[0,+\infty)$ satisfies hypothesis $(H_0)$ if there is a sequence $\{\bar x_n\}_n\subset X$ such that
\[ \tag{$H_0$}\label{eq: hyp} \sum_{n=0}^\infty \ell(\bar x_n)d(\bar x_n,\bar x_{n+1})<+\infty~\qquad\text{and }\qquad \lim_{n\to\infty} \ell(\bar x_n)=0.\]
\end{assumption} 

In the following proposition, we study the existence of supersolutions of equation~\eqref{eqn: Gs}. Observe that there is always a subsolution, namely the constant function equal to $0$. 

\begin{proposition}\label{prop.super}
    The following assertions hold true:
    \begin{enumerate}
        \item[$(a)$] If $\lambda >0$, then \eqref{eqn: Gs} always admits a supersolution.
        \item[$(b)$] If $\ell$ satisfies \eqref{eq: hyp}, then \eqref{eqn: G0} admits a supersolution.
    \end{enumerate}
\end{proposition}
\begin{proof}
    $(a)$ Since $\ell$ is a lsc function with $\textstyle\inf_X \ell=0$, the function $u:X\to \R\cup\{+\infty\}$ defined by $u=\ell/\lambda$ is a supersolution of~\eqref{eqn: Gs}.\\
    \noindent $(b)$ We split the argument into two cases. Assume that the set $[\ell=0]$ is nonempty (and therefore~\eqref{eq: hyp} is trivially satisfied). 
    Then, it follows that the function $u: X \to \R\cup\{+\infty\}$ defined by 
    \[u(x):= \begin{cases}
        0&\quad \text{ for all }x\in [\ell=0]\\
        +\infty&\quad \text{otherwise.}
    \end{cases}
    \]
    is a supersolution of~\eqref{eqn: G0}.
    Assume now that  $[\ell=0]=\emptyset$ and consider $\{\bar x_n\}_n\subset X$ be a sequence given by $(H_0)$.
    Without loss of generality, we can assume that $\{\bar x_n\}_n$ is injective. 
    Also, since $\ell$ is lsc and $\ell(x)>0$ for all $x\in X$, the sequence $\{\bar x_n\}_n$ does not have accumulation points. 
    Consider the function $u:X\to\R$ defined by 
    \[u(x):=\begin{cases}
        \sum_{n=k}^\infty \ell(\bar x_n)d(\bar x_n, \bar x_{n+1}),&~\text{if }x= \bar x_k\\
        +\infty,&~\text{otherwise}.
    \end{cases}\]
    
    We show that $u$ is lsc, $\textstyle\inf_X u=0$ and $G[u]\geq \ell$ in  $X$.
    Indeed, since $\{\bar x_n\}_n$ has no accumulation points, $\textup{dom}\,u$ is a discrete set and therefore $u$ is lsc.
    It is clear that $u\geq 0$ and that, since the series given by assumption~\eqref{eq: hyp} is convergent, $\lim_{n\to\infty} u(\bar x_n)=0$.
    It only rests to compute $G[u]$.
    Let $x\in X$.
    If $x\notin \{\bar x_n:~n\in\N\}$, then $u(x)=+\infty$. Hence $G[u](x)=+\infty$.
    If $x=\bar x_k$, for some $k\in K$, we have
    \[G[u](\bar x_k)\geq \dfrac{u(\bar x_k)-u(\bar x_{k+1})}{d(\bar x_k,\bar x_{k+1})}= \ell(\bar x_k).\]
    So, $G[u]\geq \ell$ on $X$. 
\end{proof}

Let $u: X \to \R \cup \{ + \infty \}$ be a function such that $\lambda u(x) + G[u](x) \leq \ell(x)$ for some $x \in X$. Then, we have that 
    \begin{align*}
        u(x) \leq \dfrac{u(y) + \ell(x)d(x, y)}{1 + \lambda d(x, y)}, \quad \text{ for all } y \in X.
    \end{align*}
Motivated by above inequality, we formally define the operator ${T_{\lambda, \ell}:[-\infty,+\infty ]^X\to[-\infty,+\infty ]^X}$ by
\begin{equation}\label{oper-T}
T_{\lambda,\ell} u(x):= \inf\left\{\dfrac{u(y)+\ell(x)d(x,y)}{1+\lambda d(x,y)}:~y\in X\right \}.
\end{equation}
If there is no risk of confusion, we may drop the subindexes $\lambda$ and $\ell$ and just write $T:=T_{\lambda,\ell}$.

\medskip

Observe that, for any function $u:X\to \R\cup\{+\infty\}$, we have that  $u\geq Tu$. 
Also, it easily follows that, if $\textstyle\inf_X u=0$, then $ \textstyle\inf_X Tu=0$.
That is, $T:[0,+\infty]^X\to [0,+\infty]^X$.
Albeit simple, let us proceed with the following proposition.

\begin{proposition}\label{prop: T subsolution}
    Let $u: X\to\R$ be a function.
    Then, $Tu=u$ if and only if $\lambda u(x)+G[u](x)\leq\ell(x)$ for all $x\in X$. 
    Moreover, $u$ is a subsolution of~\eqref{eqn: Gs} if and only if $Tu=u$ and $\inf_X u =0$.
\end{proposition}
\begin{proof}
    The first equivalence immediately follows from the definition of $~T$. 
    The second equivalence is a consequence of the first part of the proposition and the definition of subsolutions.
\end{proof}




The next proposition shows that $T$ can be used to define a projection onto the set of subsolutions of~\eqref{eqn: Gs}.

\begin{proposition}\label{prop: Tv=v}
    Let $u_0: X \to \R\cup\{+\infty\}$ be a non-negative function. 
    Define $u_n:=Tu_{n-1}$ for $n \geq 1$.
    Then, the sequence $\{u_n\}_n$ converges pointwise to a function $v:X \to \R$ that satisfies $Tv=v$.
\end{proposition}

\begin{proof}
    Since $u_0 \geq 0$, we have that $T^{n-1}u_0\geq T^nu_0\geq 0$, for all $n\in \N$. 
    Hence, $\{u_n\}_n$ converges pointwise to some function $v$. To show that $Tv=v$, it suffices to check that $v \leq Tv$. 
    Reasoning towards a contradiction, we assume that there is $x\in X$ such that $Tv(x)<v(x)$. Therefore, there is $y\in X$ such that
    \begin{align*}
        \dfrac{v(y)+\ell(x)d(x,y)}{1+\lambda d(x,y)}<v(x)
    \end{align*}
    By definition of $v$, there is $n\in \N$ such that
    \begin{align*}
        \dfrac{u_n(y)+\ell(x)d(x,y)}{1+\lambda d(x,y)}<v(x)\leq u_n(x).
    \end{align*}
    The above inequality yields that $v(x)\leq u_{n+1}(x)<v(x)$, which is a contradiction. The proof is complete.
\end{proof}

We denote by $T^\infty:[0, + \infty]^X \rightarrow [0, + \infty]^X$ the operator defined by
\begin{align}\label{eq: Tinfty}
    T^\infty u(x):= \lim_{n\to\infty} T^n u(x).
\end{align}
Observe that, thanks to Proposition~\ref{prop: Tv=v}, the operator $T^\infty$ is well defined. 
Moreover, if $u \in [0, + \infty]^X$ is proper, then $Tu(X) \subset  \R$. 
Also, a direct computation gives us that, for each $n \in \N$, 

\begin{align}\label{eq: Tnu}
T^n u(x)=\inf\left\{\dfrac{u(x_n)}{\prod_{k=0}^{n-1} \left( 1+\lambda d(x_k,x_{k+1}) \right)} + \sum_{j=0}^{n-1} \dfrac{\ell(x_j)d(x_j,x_{j+1})}{\prod_{k=0}^{j} \left( 1+\lambda d(x_k,x_{k+1}) \right)} \right\},
\end{align}
where the infimum is taken over all finite sequences $\{ x_k \}_{k = 0}^{n} \subset X$ with $x_0 = x$. The main result of this section reads as follows.

\begin{theorem}\label{thm: existence}
    Suppose that $u$ is a supersolution of~\eqref{eqn: Gs}. 
    Then, the function $T^\infty u$ is a solution of~\eqref{eqn: Gs}.
\end{theorem}

\begin{proof}
    For the sake of brevity, we denote by $v:X\to\R$ the function defined by
    \[v(x):=T^\infty u (x) = \lim_{n\to\infty}T^nu(x).\]
    Observe that $v\geq 0$ and that $v(X)\subset \R$ (because $u$ is proper).\\
    
    \textbf{Claim 1}: \textit{$v$ is a subsolution of \eqref{eqn: Gs}}.\\    
    Indeed, by Proposition~\ref{prop: Tv=v} we have that $Tv=v$, and since $\inf_X u = 0$, we have that $\inf_X v  = 0$. Therefore, \Cref{prop: T subsolution} implies that $v$ is a subsolution. 
    
    \medskip
    
    \textbf{Claim 2}: \textit{$v$ is a supersolution of \eqref{eqn: Gs}}.\\    
    Since $v$ is a subsolution of~\eqref{eqn: Gs}, it follows that $v$ is lsc and $\inf_X v=0$. So, we only need to show that $\lambda v+G[v]\geq \ell$.
    Fix $x\in X$.
    For every $n\in \N$, consider $\{x_i^n\}_{i=0}^n\in X$ given by~\eqref{eq: Tnu} such that $x_0^n=x$ and 
    \begin{align}\label{eq: Tnu aprox}
    T^n u(x)+\frac{1}{n} \geq \dfrac{u(x_n^n)}{\prod_{k=0}^{n-1}  \left( 1+\lambda d(x_k^n,x_{k+1}^n) \right)}  + \sum_{j=0}^{n-1} \dfrac{\ell(x_j^n)d(x_j^n,x_{j+1}^n)}{\prod_{k=0}^{j} \left( 1+\lambda d(x_k^n,x_{k+1}^n) \right)}.    
    \end{align}
    For the sake of brevity, we shall write
        \begin{align*}
            A_j^n := \dfrac{1}{\prod_{k = 0}^j \left( 1 + \lambda d(x_k^n, x_{k + 1}^n) \right)}, \text{ for every } j < n.
        \end{align*}
    Note that the finite sequence $\{x_i^n\}_{i=0}^n$ depends on $n \in \N$. 
    To prove that $v$ is a supersolution, we split the analysis into two cases.
    \smallskip \\  
    \textbf{Case 1:} $v(x)=u(x)$. 
    Since $v \leq u$, we immediately get that $G[v](x)\geq G[u](x)$. 
    Thus, 
    \[\lambda v(x)+G[v](x)\geq \lambda u(x)+G[u](x)\geq \ell(x).\]
    
    \textbf{Case 2:} $v(x)<u(x)$.
    We again split the analysis into two cases. 
    To do this, let us set, for all $n \in \N$,
    \[{\epsilon_n = \max \left\{d(x,x_i^n):~i=1,...,n \right\} }.\]

    \textbf{Case 2.1:} $\textstyle\liminf_{n\to\infty} \epsilon_n=0$.
    Up to a subsequence, we assume that $\lim_{n\to\infty}\epsilon_n=0$.
    Notice that $\{x_n^n\}_n$ converges to $x$. 
    Let us assume that $\lambda = 0$. 
    Then, $A_j^n = 1$ for every $n \in \N$ and $j < n$. 
    This leads to 
    \[v(x)=\lim_{n\to\infty} T^nu(x)\geq \lim_{n\to\infty} u(x_n^n) - \dfrac{1}{n} + \sum_{j=0}^{n-1} \ell(x_j^n)d(x_j^n,x_{j+1}^n)\geq\liminf_{n\to\infty} u(x_n^n)\geq u(x),  \]
    where the last inequality follows because $u$ is lsc. 
    Therefore, $v(x)=u(x)$, which contradicts the assumption of Case~$2$.
    Now, assume that $\lambda > 0$.
    Observe that, thanks to a telescopic sum, we have
    \begin{align}\label{eq: tri equality}
        \sum_{j=0}^{n-1} A_j^n d(x_j^n,x_{j+1}^n) = \dfrac{1}{\lambda}\left( 1-A_{n - 1}^n \right). 
    \end{align}
    Let $\rho>0$.
    Recalling that $\ell$ is lsc, consider $\delta>0$ such that $\ell(y)\geq \ell(x)-\rho$ for all $y\in B_{\delta}(x)$.
    Let $n\in \N$ be large enough such that $\epsilon_n<\delta$.
    Then,
    \begin{align*}
        T^nu(x)+\dfrac{1}{n}&\geq A_{n - 1}^n u(x_n^n) + \sum_{j=0}^{n-1} \ell(x_j^n) A_j^n d(x_j^n,x_{j+1}^n) \\
        &\geq A_{n - 1}^n u(x_n^n) + \dfrac{\ell(x)-\rho}{\lambda} \left(1-A_{n - 1}^n \right)\\
        &\geq \dfrac{\ell(x)-\rho}{\lambda} + A_{n - 1}^n \dfrac{\lambda u(x_n^n)-\ell(x)+\rho}{\lambda},
    \end{align*}
    where in the second line we have used~\eqref{eq: tri equality}.
   We define two sets of indexes
        \begin{align*}
            J_1 := \left\{ n \in \N: \lambda u(x_n^n) \geq \ell(x) - \rho \right\} \text{ and } J_2 := \N \setminus J_1.
        \end{align*}
   Since $J_1 \cup J_2 = \N$, at least one of these sets has countably infinitely many elements. If $J_2$ has countably infinitely many elements, using the fact that $A_{n - 1}^n \leq 1$ and $\lambda u(x_n^n) < \ell(x) - \rho$ for every $n \in J_2$, we have
   \begin{equation}
       T^nu(x)+\dfrac{1}{n}\geq \dfrac{\ell(x)-\rho}{\lambda}+\dfrac{\lambda u(x_n^n)-\ell(x)+\rho}{\lambda}, \text{ for all } n \in J_2.
   \end{equation}
   Therefore, taking limit inferior as $n \rightarrow \infty$ and using the lower semicontinuity of $u$, we get 
   \[ v(x) \geq \liminf_{n \to \infty} u(x_n^n) \geq u(x), \] which contradicts the assumption of Case 2. Hence, the set $J_2$ has finitely many elements and so $J_1$ has infinitely many elements. 
   We then directly get
    \begin{align*}
        T^nu(x)+\dfrac{1}{n}\geq \dfrac{\ell(x)-\rho}{\lambda}, \text{ for all } n \in J_1.
    \end{align*}
    Sending $n$ to infinity and then $\rho$ to $0$, we obtain that $\lambda v(x)\geq \ell(x)$. 
    Since $G[v](x)\geq 0$, it follows that $v$ is a supersolution of~\eqref{eqn: Gs}.\\

    \textbf{Case 2.2:} $\textstyle\liminf_{n \rightarrow \infty} \epsilon_n>0$. 
    Let $r>0$ be such that $\textstyle\liminf_{n \to \infty} \epsilon_n >r$. 
    Let $N\in \N$ be large enough such that $\epsilon_n>r$ for all $n\geq N$.
    Fix $n\geq N$.
    Thanks to the lower semicontinuity of $\ell$, let $\rho >0$ and ${\delta\in(0,r)}$ be such that $\ell(y)>\ell (x)-\rho$ for every $y \in B_{\delta}(x)$. 
    Consider now 
    \[m_n:=\inf \left\{ m\in\{1,..,n\}:~ d(x,x_m^n) > \delta \right\}.\]
    Note that, by the definition of $\delta$, we have that $1 < m_n \leq n$.
    Fix $\sigma>0$.
    Followed by \eqref{eq: Tnu aprox}, for $n$ large enough we have
    \begin{align}\notag
    v(x)+\sigma\geq & ~ T^n u(x)+\frac{1}{n}\\ \label{eq: 3.7}
     \geq & ~\left ( A_{n - 1}^n u(x_n^n) + \sum_{j=m_n}^{n-1} \ell(x_j^n)A_j^nd(x_j^n,x_{j+1}^n) \right) + \sum_{j=0}^{m_n-1} \ell(x_j^n)A_j^n d(x_j^n,x_{j+1}^n).
    \end{align}
    Let us assume that $\lambda=0$. Thus, $A_j^n =1$ for all $j<n$ and therefore, combining  inequality~\eqref{eq: 3.7}, triangle inequality and the definition of the operator $T$, we deduce
    \begin{align*}
        v(x)+\sigma \geq T^{n-m_n}u(x_{m_n})+ (\ell(x)-\rho)d(x,x_{m_n})\geq v(x_{m_n})+ (\ell(x)-\rho)d(x,x_{m_n}).
    \end{align*}
    So, rearranging the terms of the above inequality and recalling that $d(x,x_{m_n})\geq \delta$, we get
    \begin{align*}
        G[v](x)+\dfrac{\sigma}{\delta} \geq \dfrac{v(x)-v(x_{m_n})}{d(x,x_{m_n})}+\dfrac{\sigma}{d(x,x_{m_n})} \geq  \ell(x)-\rho.
    \end{align*}
    So, after sending $\sigma$ to $0$ and then sending $\rho$ to $0$, we conclude that $G[v](x)\geq \ell(x)$. 
    This shows that $v$ is a supersolution of~\eqref{eqn: G0}.

    \medskip
    
    Assume now that $\lambda>0$. From \eqref{eq: tri equality} and \eqref{eq: 3.7}, we get
    \begin{align}\notag
     v(x)+\sigma\geq &  ~ A_{m_n - 1}^n T^{n-m_n} u(x_{m_n}) + \frac{\ell(x)-\rho}{\lambda}\left(1 - A_{m_n - 1}^n \right)\\ \label{eq: 3.8}
     \geq & ~ A_{m_n - 1}^n v(x_{m_n}) + \frac{\ell(x)-\rho}{\lambda}\left(1- A_{m_n - 1}^n \right).   
    \end{align}
    Since $A^n_{m_n - 1} < 1$, above inequality leads to
    \begin{align}\label{eq: hola}
    \dfrac{\lambda}{1- A_{m_n - 1}^n} (v(x)-v(x_{m_n}))\geq \ell(x)-\rho -\lambda v(x_{m_n})-\dfrac{\lambda \sigma}{1- A_{m_n - 1}^n}.
    \end{align}
    Observe that 
    \[A_{m_n - 1}^n \leq \dfrac{1}{1+\lambda d(x,x_{m_n}^n)}. \]
    Hence, we have the estimates
    \begin{align}\label{eq: chao}
        \dfrac{1}{\lambda} (1 - A_{m_n - 1}^n) \geq \dfrac{d(x, x^n_{m_n })}{1 + \lambda d(x, x^n_{m_n })},\quad \text{and so }\quad \dfrac{\lambda}{1 - A_{m_n - 1}^n} \leq \lambda + \dfrac{1}{\delta}.
    \end{align}
    We need to divide again the argument into two cases.\\
    
    \textbf{Case 2.2.1:} $v(x)<v(x_{m_n}^n)$ for infinitely many $n\in \N$. 
    Let $n\in \N$ for which $v(x)<v(x_{m_n}^n)$.
    Then, since $A_{m_n - 1}^n<1$, inequality~\eqref{eq: 3.8} becomes
    \begin{align*}
        v(x)+\dfrac{\sigma}{1-A^n_{m_n-1}} \geq \dfrac{\ell(x)-\rho}{\lambda}.
    \end{align*}
    Plugging in the inequality~\eqref{eq: chao}, we obtain
    \begin{align}~\label{eq: 3.11}
        v(x)+\sigma\left(1+\dfrac{1}{\lambda\delta} \right)\geq \dfrac{\ell(x)-\rho}{\lambda}.
    \end{align}
    Noting that we can set $n$ as large as we want, 
    inequality~\eqref{eq: 3.11} is valid for all $\sigma >0$. After sending $\sigma$ to $0$, we have
        \begin{equation}\label{v.01}
            v(x) \geq \dfrac{\ell(x) - \rho}{\lambda}. 
        \end{equation}
    
    \medskip

    \textbf{Case 2.2.2:} $v(x)<v(x_{m_n}^n)$ for finitely many $n\in \N$.
    So, we can (and shall) assume that $n$ is large enough such that $v(x)\geq v(x_{m_n}^n)$.
    Using inequalities~\eqref{eq: hola} and~\eqref{eq: chao}, we deduce that
    \begin{equation}\label{Iqn.03}
        \left( \lambda + \dfrac{1}{d(x, x^n_{m_n })} \right) \left( v(x) -
        v(x^n_{m_n }) \right) \geq \ell(x) - \rho - \lambda v(x^n_{m_n }) - \sigma \left(\lambda +\dfrac{1}{\delta} \right).
    \end{equation} Rearranging terms in the inequality \eqref{Iqn.03}, we finally obtain
    \begin{equation}\label{eq: 3.13}
        \lambda v(x) + G[v](x) \geq \ell(x) - \rho - \sigma\left(\lambda +\dfrac{1}{\delta} \right).
    \end{equation}

Sending $\sigma$ to $0$, we conclude that 
    \begin{equation}\label{v.02}
        \lambda v(x) + G[v](x) \geq \ell(x) - \rho.
    \end{equation}
Combining Case 2.2.1 and Case 2.2.2, that is, \eqref{v.01} and \eqref{v.02}, we obtain
    \begin{align*}
        \lambda v(x) + G[v](x) \geq \ell(x) - \rho, 
    \end{align*}
independently of the behavior of the sequence $\{ v(x^n_{m_n}) \}_n$. Therefore, we can send $\rho$ to $0$ and obtain that $v$ is a super solution. 
The proof of Theorem~\ref{thm: existence} is complete.
\end{proof}
\begin{remark}~{}{\empty}\label{rmk: projection}\normalfont Proposition~\ref{prop: Tv=v} and Theorem~\ref{thm: existence} can be rephrased as follows: the operator $T^\infty$ defined on the spaces functions $f:X\to\R$ such that $\inf_X f=0$ is a projection onto the set of subsolutions of~\eqref{eqn: Gs}. Moreover, the set of supersolutions of~\eqref{eqn: Gs} is projected onto the set of solutions of the same equation.
\end{remark}

The next corollary collects the results about the existence of solutions of equation~\eqref{eqn: Gs}. Its proof follows directly from Proposition~\ref{prop.AC-seq}, Proposition~\ref{prop.super} and Theorem~\ref{thm: existence}. 

\begin{corollary}\label{thm.exists.full}
    Let $X$ be a metric space and let $\ell:X\to [0,+\infty)$ be a lsc function such that $\textstyle\inf_X \ell = 0$. The following assertions hold true
        \begin{itemize}
            \item[$(i)$] if $\lambda > 0$,  then equation~\eqref{eqn: Gs} has at least one solution  in $\LSC(X)$;
            \item[$(ii)$] if $(H_0)$ holds, then equation~\eqref{eqn: G0} has at least one solution in  $\LSC(X)$.
        \end{itemize}

    In addition, if $X$ is complete, then equation~\eqref{eqn: G0} has a solution if and only if the assumption $(H_0)$ holds.
\end{corollary}

In the case when $\lambda = 0$ and under the assumption~\eqref{eq: hyp}, we can provide an explicit formula of a solution of~\eqref{eqn: G0}. 
In Section~\ref{sec: viscosity} we show that this is the largest solution of~\eqref{eqn: G0}. 

\begin{proposition}\label{prop.solu_series}
Suppose that $(H_0)$ is fulfilled. Define the function $u_0: X \rightarrow [0, + \infty)$ as follows 
    \begin{equation}\label{eqn.series}
        u_0(x) := \inf \left\{ \sum_{n = 0}^{\infty} \ell(x_n)d(x_n, x_{n + 1}): \{ x_n \}_n \in \mathcal{K}(x) \right\}, \text{ for all } x \in X,
    \end{equation}
where $\mathcal{K}(x) := \{ \{ x_n \}_n \subset X : x_0 = x \text{ and } \lim_{n \rightarrow \infty} \ell(x_n) = 0 \}$. Then, $u$ is a solution of~\eqref{eqn: G0}. 
Moreover, one can replace $\mathcal{K}(x)$ in~\eqref{eqn.series} by
    \begin{align*}
        \widehat{\mathcal{K}}(x) := \{ \{ x_n \}_n \subset X : x_0 = x \text{, $\{ \ell(x_n) \}_n$ is decreasing and } \lim_{n \rightarrow \infty} \ell(x_n) = 0 \}.
    \end{align*}  
\end{proposition}

\begin{proof} Due to the assumption \eqref{eq: hyp}, it follows that $u_0(x)$ is well defined for every $x \in X$.  Let $\{ \bar x_n \}_n$ be a sequence given by \eqref{eq: hyp}. By definition of $u_0$, we know that
    \begin{align*}
        u_0(\bar x_n) \leq \sum_{k=n}^{\infty } \ell(\bar x_k) d(\bar x_k, \bar x_{k + 1}), \text{ for all } n \in \N,
    \end{align*}
which leads to $\inf_X u_0 = 0$. On the other hand, for any fixed $x\in X$, we note that 
\begin{align*}
        u_0(x) &= \inf \left\{ \sum_{n = 0}^{\infty} \ell(x_n)d(x_n, x_{n + 1}): \{ x_n \}_n \in \mathcal{K}(x) \right\} \\
        &=\inf \left\{ u(x_1)+\ell(x)d(x,x_1):~ x_1\in X \right\} = Tu_0(x). 
    \end{align*}
    
So, thanks to Proposition~\ref{prop: T subsolution}, $u_0$ is a subsolution of~\eqref{eqn: G0}. Lastly, replicating the computations made in the proof of Theorem~\ref{thm: existence}, Claim 2, for $\lambda=0$, we deduce that $u_0$ is a supersolution and hence it is a solution of equation~\eqref{eqn: G0}. 

\medskip

It remains to check that one can replace $\mathcal{K}(x)$ by $\widehat{\mathcal{K}}(x)$ in the definition of  $u$. Set the function 
\begin{align*}
    \widehat{u}(x) = \inf \left\{ \sum_{n = 0}^{\infty} \ell(x_n)d(x_n, x_{n + 1}): \{ x_n \}_n \subset \widehat{\mathcal{K}}(x) \right\}.
\end{align*}
By definition of infimum, we have $\widehat{u} \geq u_0$. Fix $x \in X$. We show that $u_0(x) \geq \widehat{u}(x)$. Let $\varepsilon > 0$. By definition of infimum, there exists $\{ x_n \}_n \in \mathcal{K}(x)$ such that
    \begin{equation}\label{eqn.u_approx}
        u_0(x) > \sum_{n = 0}^\infty \ell(x_n) d(x_n, x_{n + 1}) - \varepsilon.
    \end{equation}
Set $n_0  = 0$ and inductively define 
    \begin{align*}
        n_k := \min \left\{ k > n_{k - 1}: \ell(x_k) < \ell(x_{n_{k - 1}}) \right\}, \quad \text{ for all } k \geq 1.
    \end{align*}
Since $\lim_{n \to \infty} \ell(x_n) = 0$, the sequence $\{ n_k \}_k$ is well defined and strictly increasing. On one hand, we have $\{ x_{n_k}\}_k \in \widehat{\mathcal{K}}(x)$. On the other hand, using triangle inequality, we deduce
    \begin{equation}\label{eqn.n_k}
        \sum_{j = n_{k - 1}}^{n_k - 1} \ell(x_j)d(x_j, x_{j + 1}) \geq \ell(x_{n_{k-1}}) \sum_{j = n_{k - 1}}^{n_k - 1} d(x_j, x_{j + 1}) \geq \ell(x_{n_{k-1}}) d(x_{n_{k - 1}}, x_{n_k}).
    \end{equation}
Combining~\eqref{eqn.u_approx}--~\eqref{eqn.n_k} and definition of $\widehat{u}$, we get
    \begin{align*}
        u_0(x) > \sum_{k = 0}^{\infty} \ell(x_{n_{k-1}}) d(x_{n_{k - 1}}, x_{n_k}) - \varepsilon \geq \widehat{u}(x) - \varepsilon.
    \end{align*}
Sending $\varepsilon$ to $0$, we conclude that $u_0(x) \geq \widehat{u}(x)$. This completes the proof.
\end{proof}

    
\subsection{Pertinence of the definition}\label{subsec: pertinence} Since $\ell$ is real-valued, any subsolution of~\eqref{eqn: Gs} is lsc, however, the inequality $\lambda u + G[u] \geq \ell$ does not yield any regularity of $u$. In this subsection, we justify why we have chosen the space of lsc functions for the notion of supersolution. 
In the following example, we show that there are a metric space $X$ and a nonnegative function $\psi:X\to\R$ such that 
$\lambda \psi+ G[\psi]\geq \ell$, but $T^\infty \psi$ is not a solution of~\eqref{eqn: Gs}, for any $\lambda \geq 0$. Consequently, \Cref{thm: existence}
does not hold in general for functions that only satisfy $\lambda u + G[u] \geq \ell$. 
\begin{example}\normalfont     
Let $\phi:[0,1]\to[0,1]$ be an everywhere surjective function, i.e. $\phi(U)=[0,1]$ for every nonempty open set $U\subset [0,1]$  
(for instance, see \cite[p. 90]{L1904}). 
Also, consider $f:[0,1]\to (0,1]$ defined by $f(0)=1$ and $f(x)=x$, if $x>0$. 
Therefore, the function $\psi:[0,1]\to[0,1]$ defined by 
\[\psi(x):=\begin{cases} 0&~\text{if }x=0,\\
                    f\circ \phi (x) &~\text{if }x>0,
\end{cases}\]
satisfies $\psi(U)=(0,1]$ for any open interval $U\subset (0,1]$. Hence, $G[\psi](0)=0$ and $G[\psi](x)=+\infty$ for all $x\in (0,1]$. 
Let $\ell:[0,1]\to [0, + \infty)$ be any lsc function such that $\ell(0)=0$ and $\ell(x)>0$ for all $x>0$ (for instance $\ell(x)=x$). 
So, for any $\lambda\geq 0$ we have
\begin{align*}
    \begin{dcases}
        \lambda \psi + G[\psi]\geq \ell,&~\text{ on } [0,1],\\
        \inf_{x\in [0,1]} \psi(x)=0.&
    \end{dcases}
\end{align*}
    Clearly $\psi$ is not lsc. 
    Moreover, since $\psi\geq 0$, we have that $\psi\geq T^\infty \psi \geq 0$.
    Thanks to Proposition~\ref{prop: T subsolution} and Proposition~\ref{prop: Tv=v}, $T^\infty \psi$ is a subsolution of~\eqref{eqn: Gs}. 
    So, $T^\infty \psi$ is lsc and we deduce that \[0=\mathrm{lsc}(\psi)\geq T^\infty \psi \geq 0,\] where $\mathrm{lsc}(\psi)$ denotes the lsc envelop of $\psi$. 
    Thus, $T^\infty \psi$ is not a solution of~\eqref{eqn: Gs}.
\end{example}

\subsection{Uniqueness}\label{sec: uni}

In this subsection, we provide partial results about the uniqueness of solutions of equation~\eqref{eqn: Gs}. To ensure this, it is crucial to assume that $X$ is a complete metric space. 
The proof of the main result of this subsection, \Cref{thm.CP}, is based on transfinite induction and it follows the ideas of~\cite{DLS_2023, DMS_2022}.
Let us start with the following result, which is a direct consequence of \Cref{prop: 21}, Theorem~\ref{thm: existence} and \cite[Corollary 4.1]{TZ_2023}. 

\begin{corollary}\label{cor: TZDSL}
    Let $X$ be a complete metric space and let $\ell \in \LSC(X)$ be a locally bounded function satisfying $\inf_X \ell = 0$. 
    Then~\eqref{eqn: G0} admits a unique continuous solution. 
\end{corollary}

A substantial improvement of the above corollary is given in Corollary~\ref{corol.unique.improve}, in which we consider $\ell: X \to [0 ,  + \infty]$ such that $\inf_X \ell = 0$. 
Besides, as a consequence of Corollary~\ref{cor: TZDSL}, we have the following characterization of complete metric spaces.

\begin{corollary}\label{cor: characterization}
    Let $X$ be a metric space and let $x_0\in X$. 
    Consider $\ell:X\to[0, +\infty)$ defined by $\ell(x_0) = 0$ and $\ell(x)=1$ for every $x\neq x_0$. Then, $X$ is a complete metric space if and only if the equation~\eqref{eqn: G0}, with data $\ell$, has unique solution, namely, $u_0(x)=d(x,x_0)$. 
\end{corollary}
\begin{proof}
    Indeed, if $X$ is complete then the result follows from Corollary~\ref{cor: TZDSL}. 
    On the other hand, assume that $X$ is not complete. 
    Let $\{ y_n \}_n \subset X$ be a Cauchy sequence, which is not convergent. Consider the following two functions
        \begin{align*}
            u(x) = d(x, x_0), \quad \text{ and } \quad v(x) = \min \{ d(x, x_0), \lim_{n \to \infty} d(x, y_n) \}.            
        \end{align*}
    It readily follows that $u$ and $v$ are solution of~\eqref{eqn: G0} associated with $\ell$ and that $u \neq v$. 
\end{proof}

In what follows we tackle the case $\lambda > 0$. We show the uniqueness of solution of \eqref{eqn: Gs} when $X$ is a complete metric space and $\ell$ is bounded.
Let us first illustrate our ideas in the compact setting. 

\begin{proposition}\label{prop.CP.compact}{\normalfont(Comparison principle: compact case)}
Suppose that $X$ is compact, $\lambda > 0$ and that $\ell \in \LSC(X)$ is bounded and that $\inf_X \ell = 0$. Then, the following assertions hold true:
    \begin{itemize}
        \item[$(i)$] if $u, v \in \LSC(X)$ are respectively bounded sub and supersolution of~\eqref{eqn: Gs}, then $u \leq v$;
        \item[$(ii)$] equation~\eqref{eqn: Gs} admits a unique (Lipschitz) solution.
    \end{itemize}
\end{proposition}

\begin{proof} 
$(i)$ Set $M = \textstyle\sup_{x \in X} (u - v)(x)$. We aim to prove $M \leq 0$. 
Arguing by contradiction, assume that $M > 0$. 
Let us rescale the function $v$ and consider $\bar v(x) := (1 + \epsilon) v(x)$, where $\epsilon > 0$ is fixed. Also, consider 
    \begin{align*}
        \bar M = \bar M(\varepsilon):= \sup_{x \in X} ~ (u - \bar v)(x).
    \end{align*}
    
Up to shrinking $\varepsilon>0$, we can and shall assume that $\bar M \geq M/2 > 0$. 
Noting that $X$ is compact, $u$ is Lipschitz and $-v$ is upper semicontinuous, we can find $\bar x \in X$ such that $\bar M = (u - \bar v)(\bar x)$. 
Since $M>0$ and that $u = 0$ on $[\ell=0]$, we immediately obtain that $\bar x \notin [\ell=0]$.
By definition of $G[\bar v](\bar x)$ and that $\lambda \bar v(x) + G[\bar v](x)> \ell(x)$, there exists $\bar y \neq \bar x$ such that
    \begin{equation}\label{eqn.barx-bary}
        \lambda \bar M + \dfrac{(u(\bar x) - u(\bar y))_+ - (\bar v(\bar x) - \bar v(\bar y))_+}{d(\bar x, \bar y)} < 0.
    \end{equation} 
If $\bar v(\bar x) < \bar v(\bar y)$, we obtain that $\bar M < 0$, which is a contradiction.
If $\bar v(\bar x) \geq \bar v(\bar y)$, after direct computations from~\eqref{eqn.barx-bary} we obtain that $(u-\bar v)(\bar x) < (u - \bar v)(\bar y)$, which is impossible due to the definition of $\bar x$. 
Since $\bar{M}$ cannot be positive, we have that $M\leq 0$. 

$(ii)$ Note that, since $\ell$ is bounded, any subsolution (and therefore any solution) is bounded and Lipschitz.
Indeed, if $u$ is a subsolution of~\eqref{eqn: Gs}, then $u\leq \ell/\lambda$ and $G[u]\leq \ell$. 
So, after applying the comparison principle twice, we immediately obtain the uniqueness result.
\end{proof}
The next theorem shows that we can get exchange compactness by completeness in Proposition~\ref{prop.CP.compact}.
\begin{theorem}\label{thm.CP}{\normalfont(Comparison Principle: general case)} Suppose that $X$ is a complete metric space, $\lambda > 0$, $\ell \in \LSC(X)$ is bounded and $\inf_X \ell = 0$. 
If $u, v \in \LSC(X)$ are respectively bounded sub-- and supersolution of~\eqref{eqn: Gs}, then one has $u \leq v$ in $X$.    
\end{theorem}

As a direct consequence of the above theorem, we have the following uniqueness result. Its proof follows as the proof of~\Cref{prop.CP.compact} $(ii)$.

\begin{corollary}\label{cor: uniqueness lambda}
Suppose that $X$ is a complete metric space, $\lambda > 0$ and $\ell \in \LSC(X)$ is bounded with $\textstyle\inf_X \ell = 0$. 
Then, there exists a unique solution of~\eqref{eqn: Gs}.
Moreover, the solution is bounded and Lipschitz.
\end{corollary}

\begin{proof}[Proof of Theorem \ref{thm.CP}]
Arguing by contradiction, assume that there exists $x_0 \in X$ such that $m_0 := (u - v)(x_0) > 0$.
Since $u$ takes the value $0$ on the (possible empty) set $[ \ell = 0 ]$, we deduce that $x_0\notin[ \ell = 0 ]$.
In what follows, we construct by transfinite induction a generalized sequence $\{ x_\alpha \}_{\alpha}$ that satisfies the following properties: 

\medskip

Set $m_\alpha:=(u-v)(x_\alpha)$, for any $\alpha$. For any $\alpha\leq \beta$, we have
\begin{enumerate}
    \item[(P1)] $v(x_\beta) < v(x_\alpha)$;
    \item[(P2)] $m_\alpha<m_\beta$;
    \item[(P3)] $\frac{1}{2}\lambda m_0 d(x_\alpha,x_\beta)\leq m_\beta-m_\alpha$.
\end{enumerate}
Base step, $\alpha=0$: $x_0$ is already defined.

\medskip

Inductive step. \textbf{Case 1}: \textit{$\alpha$ is a  successor ordinal}. Assume that $\alpha= \xi+1$, where $\xi$ is an ordinal and the point $x_\xi$ has already been defined. If $x_\xi \in [\ell = 0]$, we stop the induction. 
Let us assume that this is not the case.
Fix $\epsilon_\xi < \min \{ \ell(x_\xi), \lambda m_0 /2 \}$. By definition of $G[v](x_\xi)$, there exists $x_{\xi+1} \neq x_\xi$ such that
\begin{equation}
    \lambda v(x_\xi) + \dfrac{(v(x_\xi) - v(x_{\xi+1}))_+}{d(x_\xi, x_{\xi+1})} > \ell(x_\xi) - \epsilon_\xi,
\end{equation}
and
\begin{equation}
    \lambda u(x_\xi) +  \dfrac{(u(x_\xi) - u(x_{\xi+1}))_+}{d(x_\xi, x_{\xi+1})} \leq \ell(x_\xi).
\end{equation}
Hence, we obtain
\begin{equation}\label{eqn.01}
    \lambda m_\xi + \dfrac{(u(x_\xi) - u(x_{\xi+1}))_+ - (v(x_\xi) - v(x_{\xi+1}))_+}{d(x_\xi, x_{\xi+1})} < \epsilon_\xi.
\end{equation}
If $v(x_{\xi+1}) \geq v(x_\xi)$, the inequality \eqref{eqn.01} leads to $\lambda m_0 < \epsilon_\xi$, which is impossible because of the choice of $\epsilon_\xi$. 
This implies that $v(x_{\xi+1}) < v(x_\xi)$. 
It follows from \eqref{eqn.01} and $\epsilon_\xi < \lambda m_0/2 \leq \lambda m_\xi/2$ that
    \begin{align*}
        \dfrac{1}{2} \lambda m_\xi d(x_\xi, x_{\xi+1}) < m_{\xi+1} - m_\xi.
    \end{align*}
From here, it directly follows that $m_{\xi+1}>m_\xi$ and therefore, $(P2)$ is satisfied. $(P3)$ follows directly from the above inequality, triangle inequality and $(P2)$.

\medskip

\textbf{Case 2:} \textit{$\alpha$ is a limit ordinal}. Since the generalized sequence $\{ m_{\kappa} \}_{\kappa < \alpha} \subset \R$ converges in $\R$ (recall that it is bounded from above and strictly increasing by construction), property (P3) implies that the sequence $\{ x_{\kappa} \}_{\kappa < \alpha}$ is Cauchy. 
By the completeness of $X$, there exists some $x_{\alpha} \in X$ such that $x_{\kappa} \rightarrow x_{\alpha}$ as $\kappa \rightarrow \alpha$. 
It remains to check that the generalized sequence $\{ x_{\kappa} \}_{\kappa \leq \alpha}$ satisfies (P1)--(P3). 
Using the fact that $v \in LSC(X)$, we deduce
    \begin{align*}
        v(x_{\xi}) > \liminf_{\substack{\kappa \rightarrow \alpha \\ \kappa < \alpha}} v(x_{\kappa}) \geq v(x_{\alpha}) \quad \text{
        for all } \xi < \alpha.
    \end{align*}
Similarly, we can prove (P2) and (P3) by using induction assumptions, $u \in Lip(X)$, $v \in LSC(X)$ and the continuity of the map $d(x, \cdot)$. 

\medskip

The above construction eventually leads to a contradiction. Indeed, if the induction stops at a point $x_\alpha \in [\ell = 0]$, then $m_\alpha=0$, which contradicts $(P2)$. 
If $X$ is countable, then for large enough cardinals $\alpha < \beta$ we will have that $x_{\alpha} = x_{\beta}$ and our construction gives $v(x_{\beta}) < v(x_{\alpha})$, which is a contradiction. On the other hand, if $X$ is uncountable, denoting by $\omega_1$ the first uncountable cardinal, we would be able to construct a generalized sequence $\{x_\alpha\}_{\alpha<\omega_1}$ such that $v(x_\alpha)<v(x_\beta)$ for all $\alpha<\beta<\omega_1$, which is also a contradiction.
So, Theorem \ref{thm.CP} is proved.
\end{proof}

\section{An equivalent definition and the maximal solution}\label{sec: viscosity}
In the classical theory of viscosity solutions, one can take benefit from the smoothness of test functions to study first order HJEs. Proceeding with a similar idea, noting that we are considering equations of the form
    \begin{align*}
        H(x, u(x), G[u](x)) = 0, \text{ for every } x \in X,
    \end{align*}
it is natural to use functions that have finite global slope (at least at the reference point) as test functions. Then, one can define a notion of "viscosity solutions" for equation~\eqref{eqn: Gs}, see Definifion~\ref{def.VIS.solu}. We show in Theorem~\ref{thm.visco.solu} that viscosity solutions coincide with the ones defined previously in Definition~\ref{def.P-solu}. Besides, we employ Perron's method to define the maximal solution of equation~\eqref{eqn: Gs}, see~\Cref{cor: Perron}.


\subsection{Viscosity solutions}
Recall that $X$ is a metric space (not necessarily complete), $\lambda\geq 0$ and $\ell\in LSC(X)$, with $\inf_X\ell=0$.
\begin{definition}\label{def.VIS.solu}
A function $u: X \rightarrow [0, + \infty)$ (resp. lsc function $v :X \rightarrow [0, + \infty)$) is said to be a viscosity subsolution (resp. viscosity supersolution) of \eqref{eqn: Gs} if $\inf_X u = 0$ (resp. $\inf_X v = 0$) and for any $x \in X$ and any function $\phi: X \rightarrow \R$ (resp. $\psi:  X \rightarrow \R$) such that
    \begin{align*}
        & \phi \geq u, (\phi - u)(x) = 0 \text{ and } G[\phi](x) < + \infty\\ 
        & \text{ $($resp. $\psi \leq v$, $(\psi - v)(x) = 0$ and $G[\psi](x) < + \infty$ $)$},
    \end{align*}
one has
    \begin{equation}\label{eqn.V-ineq}
        \begin{split}
            & \lambda u(x) + G[\phi](x) \leq \ell(x)  \\
            & \text{ $($resp. $\lambda v(x) + G[\psi](x) \geq \ell(x))$}.
        \end{split}
    \end{equation}
A function $u \in \LSC(X)$ is said to be a viscosity solution $(\text{V}$--solution, for short$)$ if it is both viscosity sub-- and supersolution.
\end{definition}

In what follows, we compare the notions of solution given by \Cref{def.P-solu} and \Cref{def.VIS.solu}. 

\begin{lemma}\label{lem.eqi-sub}
A function $u: X \rightarrow [0, + \infty)$ is a V--subsolution of~\eqref{eqn: Gs} if and only if it is a subsolution.
\end{lemma}

\begin{proof}
If $u$ is a subsolution of~\eqref{eqn: Gs}, then for any $x\in X$, we have that $G[u](x) \leq \ell(x)-\lambda u(x)$. 
Therefore, 
\[u(y)\geq u(x)-(\ell(x)-\lambda u(x))d(y,x),~\text{for all }y\in X.\]
So, if $\phi:X\to\R$ is a function such that $\phi(x)=u(x)$ and $\phi\geq u$, then $G[\phi](x)\leq \ell(x)-\lambda u(x)$, yielding the inequality~\eqref{eqn.V-ineq}.
Hence, $u$ is a V--subsolution.

\medskip

Conversely, assume that $u$ is a V--subsolution. We first claim that $\text{dom} ~ G[u] = X$. 
Arguing by contradiction, assume that there exists $\bar x \in X \setminus \text{dom}~G[u]$. This implies that for any $n \in \N$, there exists $y_n \neq \bar x$ such that
    \begin{equation}\label{quotient.contra}
        \dfrac{u(\bar x) - u(y_n)}{d(\bar x, y_n)} \geq n.
    \end{equation}
Set $\varphi_n(y) = \max \{ u(\bar x) - n d(\bar x, y), u(y) \}$ for every $y \in X$. We notice that $\varphi_n \geq u$ in $X$, $(u - \varphi_n)(\bar x) = 0$ and due to \eqref{quotient.contra}, we also have $G[\varphi_n](\bar x) = n$. Using definition \ref{def.VIS.solu} with the test function $\varphi_n$, we obtain $\lambda u(\bar x) + n \leq \ell(\bar x)$, which is a contradiction for any $n > \ell(\bar x)$.\\
We show that $u$ is a subsolution. 
Let us fix $x \in X$ and consider the following two cases.
\smallskip \\
\textbf{Case 1}: $G[u](x) = 0$. Then, $u(x) = 0$ (since it is a global minimizer) and hence, $\lambda u(x) + G[u](x) \leq \ell(x)$.

\medskip 

\textbf{Case 2}: $G[u](x) > 0$. By definition of $G[u](x)$, for any $\epsilon \in (0,G[u](x))$, there exists $y_{\epsilon} \neq x$ such that
    \begin{equation}\label{eqn.k-epsi}
        \dfrac{u(x) - u(y_{\epsilon})}{d(x, y_{\epsilon})} \geq \kappa_{\epsilon} := G[u](x) - \epsilon > 0.
    \end{equation}
If we set $\phi_{\epsilon}(y) = \max \{ u(x) - \kappa_{\epsilon}d(x, y), u(y) \}$ for every $y \in X$, then we get that $\phi_{\epsilon} \geq u$ in $X$, $(u - \phi_{\epsilon})(x) = 0$ and by \eqref{eqn.k-epsi}, we have $G[\phi_{\epsilon}](x) = \kappa_{\epsilon}$. Using $\phi_{\epsilon}$ as a test function for \eqref{eqn: Gs}, we infer that 
    \begin{align*}
        \lambda u(x) + G[u](x) - \epsilon \leq \ell(x).
    \end{align*}
Letting $\epsilon \searrow 0$, we obtain the desired inequality.
Lemma \ref{lem.eqi-sub} is proved.
\end{proof}

\begin{lemma}
 \label{lem.P-V-super}
 A function $v: X \to [0, + \infty)$ is a V--supersolution of~\eqref{eqn: Gs} if and only if it is a supersolution. 
\end{lemma}

\begin{proof}
Straightforward computations show that any supersolution is a V--supersolution. It suffices to check the converse. 
Let $v: X \rightarrow [0, + \infty)$ be a V--supersolution of \eqref{eqn: Gs}. 
Arguing by contradiction, if it is not a supersolution, there exists $\bar x \in X$ such that
    \begin{align*}
        \lambda v(\bar x) + G[v](\bar x) < \ell(\bar x). 
    \end{align*}
This implies that $G[v](\bar x)$ is finite, and therefore, $v$ itself can be used as a test function. 
Therefore, one readily gets
\begin{align*}
     \lambda v(\bar x) + G[v](\bar x) \geq \ell(\bar x),
\end{align*}
 which is a contradiction.
 \end{proof}

 As a direct consequence of Lemma~\ref{lem.eqi-sub} and Lemma~\ref{lem.P-V-super}, we obtain the equivalence between V--solutions and solutions of equation~\eqref{eqn: Gs} in metric spaces.

 \begin{theorem}\label{thm.visco.solu} Let $X$ be a metric space and let $\ell \in \LSC(X)$ satisfying $\inf_X \ell = 0$. A function $u: X \to \R$ is a solution of~\eqref{eqn: Gs} if and only if it is a V--solution.
\end{theorem}

\subsection{Perron's method}

In the classical theory of viscosity solutions, Perron's method plays a central role in finding a solution for several PDEs. 
Roughly speaking, Perron's method says that the supremum of all subsolutions (of a given PDE), bounded from above by a supersolution, is a natural candidate for being a solution. In what follows, we show how this idea is hidden in the operator $T^\infty$. 

\begin{theorem}\label{thm: T characterization}
    Let $u:X\to\R$ be a function such that $\inf_X u = 0$, and $\lambda\geq 0$. 
    Let $T$ be the operator defined by~\eqref{oper-T} and $T^\infty$ defined by~\eqref{eq: Tinfty}.
    Then
    \[T^\infty u(x) = \sup\left\{v(x):~ v~\text{is a subsolution of }\eqref{eqn: Gs}~\text{and } v\leq u \right\}.\]
    In particular, if $u$ is a supersolution of \eqref{eqn: Gs}, then the solution generated by $T^{\infty}$ satisfies
    \[T^\infty u(x)=\sup \left\{v(x):~ v~\text{is a solution of }\eqref{eqn: Gs}~\text{and } v\leq u \right\}.\]
\end{theorem}
\begin{proof}
    Let us denote by $U: X\to\R$ the function defined by 
    \[U(x):=\sup \left\{v(x):~ v~\text{is a subsolution of }\eqref{eqn: Gs}~\text{and } v\leq u \right\}.\]
    Thanks to Proposition~\ref{prop: Tv=v}, $T^\infty u$ is a subsolution of~\eqref{eqn: Gs} and therefore $ {T^\infty u\leq U\leq u}$. 
    Let $v:X\to\R$ be any subsolution of $\eqref{eqn: Gs}$ such that $u\geq v$. 
    Then, for any $x\in X$
    \begin{align*}
    Tu(x)&=\inf\left\{ \dfrac{u(y)+\ell(x)d(x,y)}{1+\lambda d(x,y)}:~y\in X\right\}\\
    &\geq \inf\left\{ \dfrac{v(y)+\ell(x)d(x,y)}{1+\lambda d(x,y)}:~y\in X\right\}\\
    &=Tv(x)=v(x).
    \end{align*}
    Therefore $Tu\geq v$. 
    Inductively, we readily get that for any $n\in\N$, $T^nu\geq v$. 
    Finally, Proposition~\ref{prop: Tv=v} implies that $T^\infty u\geq v$. 
    This completes the proof of the first part of the proposition.
    \\
    The second part of the proposition follows similarly as the first part but noticing that Theorem~\ref{thm: existence} gives us that $T^\infty u$ is a solution of~\eqref{eqn: Gs}.
\end{proof}
\begin{definition}
    A function $u:X\to\R$ is called the Perron solution of~\eqref{eqn: Gs} if $u$ is a solution of~\eqref{eqn: Gs} and is pointwise maximal among the set of solutions, i.e.
    \[u_{\lambda}(x):= \sup\{v(x):~v~\text{is a solution of }~\eqref{eqn: Gs}\}.\]
\end{definition}
The following corollary establishes the existence of the Perron solution. Note that, by definition, it is unique.
In the sequel, we denote by $u_{\lambda}$ and $u_0$ the Perron solutions of~\eqref{eqn: Gs} for $\lambda > 0$ and $\lambda = 0$ respectively.
\begin{corollary}\label{cor: Perron}
Let $\lambda>0$. Then $u_\lambda:=T^\infty \frac{\ell}{\lambda}$ is the Perron solution of~\eqref{eqn: Gs}.
For the case $\lambda =0$, let us assume that~\eqref{eq: hyp} holds.
Then, the function $u_0:X\to\R$  defined by
        \[ u_0(x):=\inf\left\{\sum_{n=0}^\infty \ell(x_n)d(x_n,x_{n+1}):~\{x_n\}_n \subset X,~x_0=x,~\lim_{n \to \infty} \ell(x_n)=0\right\}, \text{ for every } x \in X
        \]
is the Perron solution of equation~\eqref{eqn: G0}. 
\end{corollary}
\begin{proof}
        Let us start with the case $\lambda >0$. Let $v$ be any solution of~\eqref{eqn: Gs}. Since $G[v]\geq 0$, it follows that $v \leq \ell/\lambda$. 
        Thanks to Theorem~\ref{thm: T characterization} we deduce that $T^\infty \frac{\ell}{\lambda}\geq v$.

        \medskip
        
        Now we consider the case $\lambda=0$. First, observe that $u_0$ is a solution of~\eqref{eqn: G0} thanks to Proposition~\ref{prop.solu_series}. 
        It suffices to check that $u_0$ is the largest one. 
        Let $v:X\to\R$ be any subsolution of~\eqref{eqn: G0} and let $x\in X$. 
        If $\ell(x)=0$, then $u_0(x)=v(x)=0$ and there is nothing to do.
        Assume now that $\ell(x)\neq 0 $. 
        Fix $\varepsilon>0$ and consider a sequence $\{x_n\}_n \subset X$ such that $x_0=x$, $\lim_{n\to\infty} \ell(x_n)=0$ and 
        \[u_0(x)+\varepsilon\geq \sum_{n=0}^\infty \ell(x_n)d(x_n,x_{n+1}).\]
        Thanks to Proposition~\ref{prop.solu_series}, we can further assume that the sequence $\{\ell(x_n)\}_n$ is decreasing. Therefore, we get
        \[\sum_{n=0}^\infty \ell(x_{n+1})d(x_n,x_{n+1})\leq \sum_{n=0}^\infty \ell(x_n)d(x_n,x_{n + 1})< + \infty.\]
        Since $v$ is a subsolution, we have that $G[v]\leq \ell$ and then
        \[\sum_{n=0}^\infty G[v](x_{n+1})d(x_n,x_{n + 1})< + \infty.\]
        
        We can apply now~\Cref{prop.TZ} and deduce that 
        \[\liminf_{n\to\infty} v(x_n)=\inf_{x\in X} v(x)=0.\]
        Let $k\in \N$ be such that $v(x_k)<\varepsilon$.
        Direct computations lead to
        \begin{align*}
            u_0(x)+\varepsilon & \geq \sum_{n=0}^{k-1}\ell(x_n)d(x_n,x_{n+1}) \\
            & = -v(x_k)+v(x_k)+\sum_{n=0}^{k-1}\ell(x_n)d(x_n,x_{n+1}) \\
            & \geq-v(x_k) +T^kv(x)\\
            &\geq -\varepsilon+v(x),
        \end{align*}
        where the third inequality follows from the identity $Tv = v$ (see Proposition~\ref{prop: T subsolution}).
        Thus, ${u_0(x)\geq v(x)- 2\varepsilon}$.
        Since $\varepsilon >0$ is arbitrary, we deduce that $u_0(x)\geq v(x)$. 
        Since $x\in X$ is arbitrary, we deduce that $u_0\geq v$. Since $v$ is an arbitrary subsolution, we conclude that $u_0$ is Perron's solution of~\eqref{eqn: Gs}. 
        Corollary~\ref{cor: Perron} is proven.
\end{proof}

\section{Stability results}\label{sec: stability}

\subsection{\texorpdfstring{$\mathcal{L}^\infty$}{}--stability with respect to the potential}
In this subsection, we discuss the stability of the Perron solution of equation~\eqref{eqn: Gs} in terms of the data $\ell$.
Albeit simple, the following example shows that we cannot expect stability in full generality.

\begin{example}\normalfont
    For $n\in \N$, let $\ell_n:\R\to[0, +\infty)$ given by $\ell_n(n^2)=0$ and $\ell_n(x)=n^{-1}$ for all $x\neq n^2$. 
    Clearly, $\{\ell_n\}_n$ converges uniformly to $\ell_\infty\equiv 0$.
    Also, the solution $u_n$ of \eqref{eqn: G0} associated with $\ell_n$ is given by $u_n(x)= n^{-1}|x-n^2|$, and $u_\infty\equiv 0$.
    However, $u_n(0)=n$ for all $n\in \N$, so $\{u_n\}$ does not converges pointwise to $u_\infty$. 
    Moreover, $u_n$ converges (uniformly in bounded sets) to the constant function equal to $+\infty$.
\end{example}

The above example shows that we need to assume the boundedness of $X$.
Let  $\ell$, $\ell_n \in \mathrm{LSC}(X)$ be such that $\inf_X \ell = \inf_X \ell_n =0$ for all $n\in\N$, and let $u,~u_n$ be the respective Perron solutions of $\eqref{eqn: Gs}$ associated with $\ell$ and $\ell_n $. 
We prove that if $X$ is bounded and $\{\ell_n\}_n$ converges to $\ell$ uniformly, then the sequence $\{u_n\}_n$ converges to $u$ uniformly.
Let us start with the following proposition.
The diameter of $X$ is denoted by 
    \begin{align*}
        \mathrm{diam}(X) := \sup \{ d(x, y): ~ x, y \in X \}.
    \end{align*}

\begin{proposition}\label{prop: perron bounded}
    Let $X$ be a bounded metric space and let $\lambda\geq 0$. 
    Let $\ell:X\to [0, + \infty)$ be a lsc function such that $\textstyle \inf_X \ell = 0$.
    Then, the function $v(x):=\ell(x)\textup{diam}(X)$ is a supersolution of~\eqref{eqn: Gs} associated with $\ell$.
    Moreover, $T^\infty_\lambda v$ is the Perron solution of~\eqref{eqn: Gs}.
\end{proposition}
\begin{proof}
    Since $\inf_X \ell = 0$, it follows that $\inf_X v = 0$. 
    So, for any $x\in X$, direct computations lead to
    \begin{align}\label{eq: G bounded}   
    G[v](x)\geq \inf_{y\neq x} \dfrac{\ell(x)\textup{diam}(X)}{d(x,y)}\geq \ell(x).    
    \end{align}
    Thus, $v$ is a supersolution of \eqref{eqn: Gs}. 
    Let $w:X\to\R$ be any subsolution of $\eqref{eqn: Gs}$. Then, $\inf_X w(x)=0$ and $G[w](x)\leq \ell(x)-\lambda w(x)\leq \ell(x)$. 
    Therefore, reasoning as in \eqref{eq: G bounded} we deduce that ${w(x)\leq \ell(x)\textup{diam}(X)}$.
    Finally, thanks to Theorem~\ref{thm: T characterization}, $T^\infty v$ is the Perron solution of equation~\eqref{eqn: Gs}.
\end{proof}

Let us proceed with the following lemma, which is needed to prove the stability of the Perron solution of \eqref{eqn: Gs} in terms of $\ell$.
Observe that if $X$ is a bounded metric space, then any lsc function $\ell:X\to[0, + \infty)$, with $\inf_X \ell=0$, satisfies~\eqref{eq: hyp}.

\begin{lemma}\label{lem: stability2}
    Let $X$ be a bounded metric space and let $\lambda\geq 0$. 
    Let $\ell, \kk :X\to [0,+\infty)$ be lsc functions such that $\inf_X \ell=\inf_X \kk =0$. 
    Let $u,v:X\to\R$ be the Perron solutions of equation \eqref{eqn: Gs} associated with  $\ell$ and $\kk$ respectively.
    Then, for any $\rho>0$ we have that
    \[v(x)-u(x)\leq (\rho+\|\kk-\ell \|_\infty)\textup{diam}(X) +\|\kk-\ell \|_\infty \dfrac{u(x)}{\rho}, \quad \text{ for all } x \in X. \]
\end{lemma}
\begin{proof}
   
    Fix $\varepsilon>0$ and $\rho>0$.
    Thanks to Proposition~\ref{prop: perron bounded}, $u$ and $v$ are bounded from above by $\textup{diam}(X)\ell$ and $\textup{diam}(X)\kk $ respectively, $u = T^\infty(\ell \mathrm{diam}(X))$ and $v = T^\infty(\mathscr{k}\mathrm{diam}(X))$. 
    Let $x\in X$. 
    Assume that $\ell(x)\leq \rho$. 
    Then, 
    \begin{align*}
        v(x)\leq \kk(x)\textup{diam}(X)\leq (\rho+\|\kk-\ell\|_\infty)\textup{diam}(X).
    \end{align*}
    From the above inequality, we get 
    \begin{align}\label{eq: ineq 1}
        v(x)-u(x) \leq (\rho+\|\kk-\ell\|_\infty)\textup{diam}(X).
    \end{align}
    Now, assume that $\ell(x)> \rho$. 
    Let $\{x_i\}_{i=0}^n\subset X$ be a sequence such that $x_0=x$ and 
    \[u(x)+\varepsilon \geq \frac{\ell(x_n)\textup{diam}(X)}{\prod_{j=0}^{n-1} 1+\lambda d(x_j,x_{j+1})}+\sum_{i=0}^{n-1}\frac{\ell(x_i)d(x_i,x_{i+1})}{\prod_{j=0}^i 1+\lambda d(x_j,x_{j+1})}.\]
    Let $I\subset\{0,...,n\}$ be the largest interval of integers containing $0$ such that $i\in I$ if $\ell(x_i)\geq \rho$. 
    Then, if $n\notin I$ we deduce that 
    \begin{align}\label{eq: bounded length 22}
    \dfrac{u(x)+\varepsilon}{\rho}\geq \sum_{i=0}^{\max(I)}\frac{d(x_i,x_{i+1})}{\prod_{j=0}^i 1+\lambda d(x_j,x_{j+1})}.    
    \end{align}
    On the other hand, if $n\in I$, we deduce that
    \begin{align}\label{eq: bounded length 3}
    \dfrac{u(x)+\varepsilon}{\rho}\geq \frac{\textup{diam}(X)}{\prod_{j=0}^{n-1} 1+\lambda d(x_j,x_{j+1})}+\sum_{i=0}^{n-1}\frac{d(x_i,x_{i+1})}{\prod_{j=0}^i 1+\lambda d(x_j,x_{j+1})}.    
    \end{align}
    \textbf{Case 1}: $\max(I)<n$.
    Then $\ell(x_{\max(I)+1})< \rho$ and $\kk(x_{\max(I)+1})<\rho+\|\kk-\ell\|_\infty$. 
    Using~\eqref{eq: bounded length 22} and Proposition~\ref{prop: perron bounded}, we compute
    \begin{align*}
        u(x)+\varepsilon&\geq \sum_{i=0}^{\max(I)}\frac{\ell(x_i)d(x_i,x_{i+1})}{\prod_{j=0}^i 1+\lambda d(x_j,x_{j+1})}\\
        &\geq \sum_{i=0}^{\max(I)}\frac{\kk(x_i)d(x_i,x_{i+1})}{\prod_{j=0}^i 1+\lambda d(x_j,x_{j+1})} + \dfrac{\kk(x_{\max(I)+1})\textup{diam}(X)}{\prod_{j=0}^{\max(I)} 1+\lambda d(x_j,x_{j+1})}\\
        &-\dfrac{(\rho+\|\kk-\ell\|_\infty)\textup{diam}(X)}{\prod_{j=0}^{\max(I)} 1+\lambda d(x_j,x_{j+1})}-\|\kk-\ell\|_\infty \sum_{i=0}^{\max(I)}\frac{d(x_i,x_{i+1})}{\prod_{j=0}^i 1+\lambda d(x_j,x_{j+1})}\\
        & \geq v(x)-(\rho+\|\kk-\ell\|_\infty)\textup{diam}(X) -\|\kk-\ell\|_\infty \dfrac{u(x)+\varepsilon}{\rho}.
    \end{align*}
    Therefore, we have that
    \begin{align}\label{eq: ineq 2}
        v(x)-u(x)\leq \varepsilon+(\rho+\|\kk-\ell\|_\infty)\textup{diam}(X) +\|\kk-\ell\|_\infty \dfrac{u(x)+\varepsilon}{\rho}.
    \end{align}

    \textbf{Case 2}: $\max(I)=n$. We can proceed as above but using~\eqref{eq: bounded length 3} instead of~\eqref{eq: bounded length 22} and obtain
    \begin{equation}\label{kk}
       v(x) - u(x) \leq \varepsilon + \|\kk-\ell\|_\infty \dfrac{u(x)+\varepsilon}{\rho}.
    \end{equation}
    
    Combining the inequalities~\eqref{eq: ineq 1}, ~\eqref{eq: ineq 2} and \eqref{kk}, we deduce that, for all $x\in X$:
    \[
    v(x)-u(x)\leq \varepsilon+(\rho+\|\kk-\ell\|_\infty)\textup{diam}(X) +\|\kk-\ell\|_\infty \dfrac{u(x)+\varepsilon}{\rho}, 
    \]
    independently of the behavior of the sequence $\{ x_n \}_n$.
    Since $\varepsilon>0$ is arbitrary, we finally get 
    \[
    v(x)-u(x)\leq (\rho+\|\kk-\ell\|_\infty)\textup{diam}(X) +\|\kk-\ell\|_\infty \dfrac{u(x)}{\rho}.
    \]
\end{proof}

Now, we are able to prove the $\mathcal{L}^\infty$--stability of the Perron solution of equation~\eqref{eqn: Gs}. 
\begin{theorem}\label{thm: stability2}
    Let $X$ be a bounded metric space and let $\lambda \geq 0$. 
    Let $\ell,\ell_n:X\to [0,\infty)$ be lsc functions such that $\inf_X \ell =\inf_X \ell_n =0$.
    Further, assume that $\ell$ is bounded.
    Denote by $u$ and $u_n$ the Perron solutions of~\eqref{eqn: Gs} associated with $\ell$ and $\ell_n$ respectively.
    Then, if $\ell_n$ converges to $\ell$ uniformly, $u_n$ converges to $u$ uniformly.
\end{theorem}

\begin{proof}
    Since $\ell$ is bounded, without the loss of generality we can assume that there is $M>0$ such that $\|\ell_n\|_\infty\leq M$ for all $n\in \N$ and $\|\ell\|_\infty\leq M$.
    Then, thanks to Proposition~\ref{prop: perron bounded}, $\|u_n\|_\infty\leq M \textup{diam}(X)$.
    Fix $\rho>0$. 
    So, applying Lemma~\ref{lem: stability2} twice, we get that for any $x\in X$ and any $n\in \N$

    \[|u_n (x)-u(x)| \leq (\rho+\|\ell-\ell_n\|_\infty \textup{diam}(X) )+ \|\ell-\ell_n\|_\infty \dfrac{\max(u(x),u_n(x))}{\rho}.\]
    Therefore
    \[\|u_n-u\|_\infty \leq (\rho+\|\ell-\ell_n\|_\infty \textup{diam}(X) )+ \|\ell-\ell_n\|_\infty \dfrac{M\textup{diam}(X)}{\rho}.\]
    Taking limit superior as $n$ tends to $+ \infty$ we get that
    \[\limsup_{n\to\infty}\|u_n-u\|_\infty \leq \rho \textup{diam}(X).\]
    Since $\rho>0$ is arbitrary, after sending $\rho$ to $0$, we show that the above limit is equal to $0$.
\end{proof}
\begin{remark}\normalfont
    A careful inspection of the proof of Lemma~\ref{lem: stability2} shows that the norm $\|\cdot\|_\infty$ can be replaced by its asymmetric version $\|\cdot|_\infty$, that is $\|f|_\infty=\sup_{x \in X} \max \{ f(x), 0 \}$. 
    This is related to the unilateral results obtained in~\cite{DD_2023}, in which the convexity of the functions somehow plays the role of boundedness of the space. 
    In the context of~\Cref{lem: stability2}, we can obtain the following \textit{unilateral} estimate:
        \begin{align*}
            v(x) - u(x) \leq \left( \rho + \| \mathscr{k} - \ell |_{\infty} \right) \text{diam}(X) + \| \mathscr{k} - \ell |_{\infty}\dfrac{u(x)}{\rho}, \text{ for all } x \in X. 
        \end{align*}
    Consequently, we get
        \begin{align*}
            \| v - u |_{\infty} \leq \left( \rho + \| \mathscr{k} - \ell |_{\infty} \right) \text{diam}(X) + \| \mathscr{k} - \ell |_{\infty}\dfrac{\| \ell \|_{\infty} \mathrm{diam}(X)}{\rho}, \text{ for all } \rho > 0.
        \end{align*}
\end{remark}

  \subsection{Convergence of the discounted solutions}\label{subsec: discount converg}
    As a consequence of \Cref{thm: stability2}, we show that if the underlying metric space is bounded, then the Perron solution of \eqref{eqn: Gs} for $\lambda>0$ converges uniformly to the solution of \eqref{eqn: G0} as $\lambda$ tends to $0$. 
    \begin{corollary}\label{cor: stability fathi}
        Let $X$ be a bounded metric space and let $\ell:X\to[0, +\infty)$ be a bounded and lsc function such that $\inf_X \ell=0$.
        For any $\lambda\geq 0$, denote by $u_\lambda$ the Perron solution of~\eqref{eqn: Gs} associated with $\lambda$. Then, the following assertions hold true. 
        \begin{enumerate}
            \item[(a)] $u_\lambda$ converges to $u_0$ uniformly as $\lambda$ tends to $0$.
            \item[(b)] If $X$ is compact and $\alpha>0$, then $u_\lambda$ converges to $u_\alpha$ uniformly as $\lambda$ tends to $\alpha$.
        \end{enumerate}
    \end{corollary}

    Let us start proving some useful facts.
    \begin{proposition}\label{prop: facts}
        Let $X$ be a metric space and $\ell:X\to [0, + \infty)$ be a lsc function such that $\inf_X\ell=0$. Under the same notation of Corollary~\ref{cor: stability fathi}, let $\lambda\geq 0$. Then:
        \begin{itemize}
            \item[(a)] For any $0\leq \alpha<\beta$, $u_\beta\leq u_\alpha $.
            \item[(b)] Assume that $\ell$ is bounded, denote by $\widetilde{X}$ the completion of $X$ and by $\widetilde{\ell}:\widetilde{X}\to [0,+\infty)$ the maximal lsc extension of $\ell$ to $\widetilde{X}$. 
            Also, denote by $v_\lambda:\widetilde{X}\to\R$ the Perron solution of~\eqref{eqn: Gs} associated with $\widetilde{\ell}$.
            Then, $u_\lambda = v_\lambda$ on $X$. 
            \item[(c)] Assume that $\ell$ is bounded. Then, $u_\lambda$ is the Perron solution of~\eqref{eqn: G0} associated with $\ell_\lambda:=\ell-\lambda u_\lambda$.
        \end{itemize}
    \end{proposition}
    \begin{proof}
        $(a)$ Observe that $\alpha u_\beta +G[u_\beta]= \ell-(\beta-\alpha) u_\lambda \leq \ell$. Therefore, $u_\beta$ is a subsolution of~$(\EuScript{G}_\alpha)$.
        So, by definition of Perron solution, we have that $u_\alpha \geq u_\beta$.

        \medskip
        
        $(b)$ Along this proof, we denote by $G_X[\cdot]$ and $G_{\widetilde{X}}[\cdot]$ the global slope operator defined for functions on $X$ and on $\widetilde{X}$ respectively. 
        Note that $\widetilde{\ell}(\widetilde{x})=\liminf_{y\to \widetilde{x},~y\in X}\ell(y)$ for all $\widetilde{x}\in \widetilde{X}$. 
        Thus, $\widetilde{\ell}$ is bounded and coincide with $\ell$ on $X$. 
        So, $u_\lambda$ and $v_\lambda$ are Lipschitz functions (see Proposition~\ref{prop: 21}). 
        Denote by $\widetilde{u}_\lambda$ the unique continuous extension of $u_\lambda$ to $\widetilde{X}$. 
        By continuity of $\widetilde{u}_\lambda$ and $v_\lambda$, it readily follows that
        \begin{align}\label{eq: inside}
        G_{\widetilde{X}}[\widetilde{u}_\lambda](x)=G_X[u_\lambda](x)\quad\text{and}\quad G_{\widetilde{X}}[v_\lambda](x)=G_X[v_\lambda|_X](x),~\text{for all }x\in X,    
        \end{align}
        where $v_\lambda|_X$ denotes the restriction of $v_\lambda$ to $X$.
        To finish the proof, it is enough to show that $\widetilde{u}_\lambda=v_\lambda$.

        \medskip
        
        \textbf{Inequality}: $\widetilde{u}_\lambda\leq v_\lambda$. Thanks to Proposition~\ref{prop: 22}, $G_{\widetilde{X}}[\widetilde{u}_\lambda]$ is a lsc function. So, \eqref{eq: inside} and the fact that $u_\lambda$ is solution of~\eqref{eqn: Gs} lead to
        \[\lambda \widetilde{u}_\lambda+ G_{\widetilde{X}}[\widetilde{u}_\lambda]\leq \widetilde{\ell}. \]
        So, $\widetilde{u}_\lambda$ is a subsolution of ~\eqref{eqn: Gs} associated with $\widetilde{\ell}$ (on $\widetilde{X}$) and then, $\widetilde{u}_\lambda\leq v_\lambda$.

        \medskip

        \textbf{Inequality}: $v_\lambda\leq \widetilde{u}_\lambda$. From~\eqref{eq: inside} we deduce that $v_\lambda|_X$ is a solution of~\eqref{eqn: Gs} associated with $\ell$. Therefore, $v_\lambda|_X\leq u_\lambda$. By continuity we finally get that $v_\lambda\leq \widetilde{u}_\lambda$.\\

        $(c)$ Thanks to part $(b)$, we know that $\widetilde{u}_\lambda$ is the Perron solution of~\eqref{eqn: Gs} associated with $\widetilde{\ell}$. 
        Therefore, $\widetilde{u}_\lambda$ is a solution of~\eqref{eqn: G0} associated with $\widetilde{\kk}_\lambda:=\widetilde{\ell}-\lambda \widetilde{u}_\lambda$.
        Since $\widetilde{X}$ is complete, there is at most one solution of equation~\eqref{eqn: G0}, see \Cref{cor: TZDSL}.
        Henceforth, $\widetilde{u}_\lambda$ is the Perron solution of~\eqref{eqn: G0} associated with $\widetilde{\kk}_\lambda$.
        Using part $(b)$ once again, we obtain that $\widetilde{u}_\lambda|_X~(=u_\lambda)$ is the Perron solution of~\eqref{eqn: G0} associated with $\widetilde{\kk}_\lambda|_X=\ell-\lambda u_\lambda=:\ell_\lambda$.
    \end{proof}
    Now we can proceed with the proof of Corollary~\ref{cor: stability fathi}.
    \begin{proof}[Proof of Corollary~\ref{cor: stability fathi}]
        $(a)$ By Proposition~\ref{prop: perron bounded}, we get that $\|u_0\|\leq \mathrm{diam}(X)\sup_X \ell = C<\infty$.
        Thanks to Proposition~\ref{prop: facts}~$(a)$, the family $\{u_\lambda\}_\lambda$ is uniformly bounded from above by $C$.
        Now, consider the functions $\ell_\lambda:=\ell-\lambda u_\lambda$.
        Since $u_\lambda$ is continuous (Proposition \ref{prop: 21}) and is the Perron solution of~\eqref{eqn: Gs}, it follows that $\ell_\lambda=G[u_\lambda]$. Therefore, $\ell_\lambda$ is lsc
        and satisfies $\inf_X \ell_\lambda =0$.  
        Note now that 
        \[\|\ell -\ell_\lambda\|_\infty=\|\lambda u_\lambda\|_\infty\leq \lambda C,~\quad \text{for all }\lambda>0. \]
        Hence, $\{\ell_\lambda\}_\lambda$ converges to $\ell$ uniformly as $\lambda$ tends to $0$. 
        Now, thanks to Theorem~\ref{thm: stability2} and Proposition~\ref{prop: facts}~$(c)$, we finally deduce that $u_\lambda$ converges to $u_0$ uniformly as $\lambda$ tends to $0$.

        \medskip

        $(b)$ For the sake of simplicity, let us assume that $\alpha=1$. We first prove the case $\lambda\to 1^-$.
        Thanks to Proposition~\ref{prop: facts}~(a), we know that the family $\{u_\lambda\}_{\lambda<1}$ is decreasing (as $\lambda$ tends to $1$) and bounded from below by $u_1$. 
        Therefore, there is a function $v:X\to\R$ such that $u_\lambda$ converges to $v$ pointwise as $\lambda$ tends to $1^-$. 
        Since $X$ is compact and $u_\lambda$ is continuous, Dini's theorem implies that this convergence is uniform.
        Due to Proposition~\ref{prop: facts}~(c), $u_\lambda$ is the Perron solution of~\eqref{eqn: G0} associated with $\ell_\lambda= \ell- \lambda u_\lambda$. 
        From here, it readily follows that $\ell_\lambda$ converges to $\widetilde{\ell}:= \ell-v$ uniformly, and therefore, thanks to Theorem~\ref{thm: stability2}, $u_\lambda$ converges uniformly to a function $w$ which is the Perron solution of~\eqref{eqn: G0} associated with $\widetilde{\ell}$. 
        By uniqueness of the pointwise convergence, we deduce that $v=w$, and therefore $v$ satisfies $v+G[v]=\ell$.
        The argument for the limit of $u_\lambda$ as $\lambda$ tends to $1^+$ is analogous. 
        Indeed, the only difference is that the family $\{u_\lambda\}_{\lambda>1}$ is increasing (as $\lambda$ goes to $1$) and bounded from above by $u_1$.

    \end{proof}
    To finish this section, we present an ergodic result that does not require any assumption on $X$. 
    
    \begin{proposition}\label{prop.lambda-to-alpha}
        Let $X$ be a metric space and let $\ell:X\to [0, + \infty)$ be a lsc function such that $\inf_X \ell=0$. 
        For any $\lambda > 0$, denote by $u_\lambda$ the Perron solution of~\eqref{eqn: Gs}. Then, for any $\alpha>0$, $u_\lambda$ converges to $u_\alpha$ pointwise as $\lambda$ tends to $\alpha^-$.
    \end{proposition}
    \begin{proof}
        For the sake of simplicity let us fix $\alpha=1$.
        Thanks to Proposition~\ref{prop: facts}~(a), the family of functions $\{u_\lambda\}_{\lambda<1}$ is decreasing as $\lambda$ tends to $1^-$ and is bounded from below by $u_1$. 
        Therefore, $u_\lambda$ converges pointwise to a function $v:X\to \R$, as $\lambda$ tends to $1^-$, and this limit satisfies $v\geq u_1$. 
        So, it suffices to check that $v\leq u_1$. 
        Let $x\in X$ and $\varepsilon>0$.
        Thanks to Corollary~\ref{cor: Perron} and Theorem~\ref{thm: T characterization}, we deduce that $u_\lambda = T^\infty_\lambda (2\ell)$, for all $\lambda\geq 1/2$. 
        Fix $\lambda>1/2$.
        Followed by the definition of $T^\infty$~\eqref{eq: Tinfty}, there is a finite sequence $\{x_i\}_{i=0}^k\subset X$, with $x_0=x$ such that
        \begin{align*}
            u_1(x)+\varepsilon \geq \dfrac{2\ell(x_k)}{\prod_{j=0}^{k-1}(1+d(x_j,x_{j+1}))}+\sum_{i=0}^{k-1}\dfrac{\ell(x_i)d(x_i, x_{i+1})}{\prod_{j=0}^i(1+d(x_j,x_{j+1}))}. 
        \end{align*}
        Observe that the right hand side of the following inequality 
        \begin{align*}
            u_\lambda (x)\leq T^k_\lambda (2\ell)(x)\leq \dfrac{2\ell(x_k)}{\prod_{j=0}^{k-1}(1+\lambda d(x_j,x_{j+1}))}+\sum_{i=0}^{k-1}\dfrac{\ell(x_i)d(x_i, x_{i+1})}{\prod_{j=0}^i(1+\lambda 
            d(x_j,x_{j+1}))},
        \end{align*}
        is continuous with respect to $\lambda$. Therefore, we deduce  
        \[v(x)=\lim_{\lambda\to 1^-}u_\lambda(x)\leq u_1(x)+\varepsilon.\]
        Since $\varepsilon>0$ is arbitrary, we get that $v(x)\leq u_1(x)$.
        The proof is now complete.
    \end{proof}
\section{Applications}\label{sec: appli}

This section is devoted to applying our previous results and techniques. 
In the first subsection, similarly to \cite{E-B_2023}, we consider an approximation scheme for the solution of the global slope equation~\eqref{eqn: G0} and also a solution of the local slope equation, motivated by the shape of the Perron solution of equation \eqref{eqn: G0}. 
Additional structure on the metric space is required to carry out this approximation. 
In the second subsection, we deal with the case in which the data $\ell$ admits the value $+\infty$. 
To the best of our knowledge, as a consequence of our analysis, we obtain a new integration formula that allows us to recover, up to a constant, lsc functions which are bounded from below in terms of their global slopes.

\subsection{An approximation scheme}\label{subsec: local-global}

We have already shown that only mild conditions on the data $\ell$ are required to construct a solution of equation~\eqref{eqn: Gs}.
However, this is not the case if we replace the global slope operator $G[\cdot]$ with the local slope operator $s[\cdot]$. 
Indeed, the classical Eikonal equation is considered whenever $X$ is an open subset of $\textstyle\R^d$, and further, its purely metric version usually requires to work in complete length spaces.
To fix ideas, let $X$ be a complete length space and let $\ell:X\to[0, + \infty)$ be a continuous function such that $[\ell=0]$ is nonempty.
Under these assumptions, it is standard that the local slope equation 
\begin{equation}\tag{$\EuScript{L}$}\label{eqn: Ls}
    \begin{dcases}
    s[V](x)=\ell(x),&~\text{for all }x\in X,\\
     V(x)=0,&~\text{for all }x\in [\ell=0],\\
    \end{dcases}
\end{equation}
admits a solution $V:X\to\R$ given by
\begin{equation}\label{eq: solution local}
    V(x):=\inf\left\{ \int_0^T \ell(\gamma(t))dt\right\},~\text{for all }x\in X.
\end{equation}
where the infimum is taken over all the $1$-Lipschitz curves starting at $x$ (i.e. $\gamma(0)=x$) and landing on $[\ell=0]$ (i.e. $\gamma(T)\in [\ell=0]$) (see \cite[Theorem 4.2]{GHN_2015} e.g.). 
A complete discussion about the uniqueness issue of equation~\eqref{eqn: Ls} can be found in~\cite{DLS_2023, DMS_2022}. 
In the forthcoming analysis, we use the following result that summarizes \cite[Theorem 4.2]{GHN_2015} and \cite[Corollary 2.5]{DS_2022}.

\begin{theorem}\label{thm: solution eikonal}
    Let $X$ be a complete length space. 
    Let $\ell: X\to [0, + \infty)$ be a continuous function such that $[\ell = 0]$ is nonempty.
    Then, the function defined by~\eqref{eq: solution local} is a solution of equation~\eqref{eqn: Ls}. Moreover, if $X$ is compact, equation~\eqref{eqn: Ls} admits only one solution.
\end{theorem}

Let us now define the $R$-semiglobal slopes $G_R$ (with $R>0$), which formally interpolate the local slope and the global slope. 
The semiglobal slopes were introduced in~\cite[Proposition 3.7]{DMS_2022} as examples of what the authors called abstract descent modulus.
For a fixed $R > 0$ and $u: X \rightarrow \R$, the $R$-\textit{semiglobal slope} of $u$ at $x \in X$ is defined by
    \begin{equation}
        G_R[u](x) := 
        \begin{dcases}
            \sup_{y \in B_R(x) \setminus \{x \}} \dfrac{(u(x) - u(y))_+}{d(x, y)}, & \text{ if } B_R(x)\setminus \{x \} \neq \emptyset, \\
            0, & \text{ otherwise}.
        \end{dcases}        
    \end{equation}
Recall that $B_R(x)$ denotes the open ball of center $x$ and radius $R$. Note that, for any $u:X\to \R$ and any $x\in X$ the following limits hold
\[ G[u](x)=\lim_{R\to\infty} G_R[u](x),~\text{and }s[u] (x)=\lim_{R\to 0} G_R[u](x).\]
Assume now that $[\ell=0]\neq \emptyset$.
Inspired by the $R$-semiglobal slope and the solution of equation~\eqref{eqn: G0}, let us formally define the function $v_R:X\to\R$ by 
\begin{align}\label{eqn: fn vR}
v_R(x)= \inf \left\{ \sum_{i = 0}^{k-1} \ell(x_i) d(x_i, x_{i + 1}): x_0 = x, x_{k} \in [\ell=0], \text{ and } d(x_i,x_{i+1})<R~\text{ for } i \in \{0, ..., k \} \right\}.
\end{align}
The above functions can be compared to the ones used in \cite[Section 1.1]{E-B_2023} in which the function $\ell$ is replaced by the so-called minimal weak upper gradients.  The following lemma implies that, if $X$ is connected, then $v_R$ is well defined. 



\begin{lemma}\label{lem.connect_G_R}
Suppose that $X$ is connected, $R > 0$ and $Y \subset X$ is a nonempty set. 
Let $\mathcal{A}$ be the set of points $x \in X$ such that there exists a finite sequence $\{ x_i \}_{i = 0}^k \subset X$ satisfying
    \begin{equation}\label{R-connect-seq}
        x_0 = x, \quad x_k \in Y\quad \text{and}\quad d(x_i, x_{i + 1}) < R \text{ for every $i \in \{ 0, \cdots, k - 1 \}$ } .
    \end{equation}
Then, $\mathcal{A} = X$. 
\end{lemma}

\begin{proof}
To check that $\mathcal{A}$ is open, note that $B_{R}(x) \subset \mathcal{A}$ for every $x \in \mathcal{A}$.
This also shows that $\mathcal{A}$ is closed.
Now, since $X$ is connected and $Y\subset \mathcal{A}$, we conclude that $\mathcal{A} = X$. 
\end{proof}

The functions $v_R$ enjoy the following property.
\begin{proposition}\label{prop: solution semiglobal}
    Let $X$ be a connected metric space, $R>0$ and $\ell:X \to [0, + \infty)$ be a lsc function such that $[\ell = 0]$ is nonempty. 
    Then, $G_R[v_R]\leq \ell$. 
\end{proposition}

\begin{proof}
Thanks to Lemma~\ref{lem.connect_G_R}, the function $v_R$ is well defined.
Let $x\in X$ and $y\in B(x,R)\setminus\{x\}$. Observe that immediately from the definition of $v(x)$ we have
\begin{align*}
    v(x)&= \inf \left\{ \sum_{i = 0}^{k-1} \ell(x_i) d(x_i, x_{i + 1}): x_0 = x, x_{k} \in [\ell=0], \text{ and } d(x_i,x_{i+1})<R~\text{ for } i \in \{0, ..., k \} \right\}\\
    &\leq \ell(x)d(x,y)+ v(y). 
\end{align*}
Since the above inequality holds true for any $y\in B(x,R)\setminus\{x\}$, we deduce that $G_R[v_R](x)\leq \ell(x)$.
\end{proof}
\begin{remark}\normalfont
    The functions $v_R$ may have their own interest, but their study is out of the scope of this paper. 
    For instance, a natural question arises: under which metric conditions of $X$ (if any) does the function $v_R$ satisfy ${G_R[v_R]=\ell}$? 
\end{remark}
In what follows, we are interested in the following natural question
\begin{center}
        \textit{What are the asymptotic behaviors of the family $\{ v_R \}_{R > 0}$ as $R$ tends to $+\infty$ and tends to $0$?}
\end{center}

In Theorem \ref{thm: local and global} we provide necessary conditions on the metric space $X$ to obtain the convergence of the family $\{v_R\}_R$ to the Perron solution of~\eqref{eqn: G0} and to the function defined by~\eqref{eq: solution local} (the solution of~\eqref{eqn: Ls}) as $R$ tends to $+ \infty$ and to $0$ respectively. 
Recall that if $[\ell=0]$ is nonempty, then the assumption~\eqref{eq: hyp} is trivially satisfied. 
\begin{theorem}\label{thm: local and global}
Assume that $\ell: X \rightarrow [0, + \infty)$ is a lsc function such that $[\ell = 0 ] \neq \emptyset$. For each ${R > 0}$, we denote by $v_R$ the function defined by~\eqref{eqn: fn vR}. 
Denote by $u_0: X \to \R$ and $V: X \to \R$ the Perron solution of \eqref{eqn: G0} and the solution of \eqref{eqn: Ls} given by \eqref{eq: solution local} both associated with $\ell$. 
Then, the following statements hold true. 
    \begin{itemize}
        \item[$(i)$] If $X$ is connected, then $v_R$ converges pointwise to $u_0$ and $R$ tends to $+\infty$.
         \item[$(ii)$] If $X$ is a complete length space and $\ell$ is continuous, then $\liminf_{R \to 0^+} v_R \leq V$ pointwise.
        \item[$(iii)$] If $X$ is a compact length space (and therefore geodesic) and $\ell$ is continuous, then $v_R$ converges uniformly to $V$ as $R$ tends to $0$. 
    \end{itemize}
\end{theorem}

\begin{proof}
First, observe that by the definition of infimum, for any $0<R_1<R_2$, we have that ${v_{R_1}\geq v_{R_2}}$. 
So, the following limits are well defined:
\begin{align*}
    v_{\infty}(x) &:= \lim_{R\to\infty} v_R(x) =\inf\left\{v_R(x):~R>0 \right\},\\
    v_0(x) &:= \lim_{R\to 0} v_R(x) =\sup\left\{v_R(x):~R>0\right\}, \quad \text{ for every } x \in X.
\end{align*}

$(i).$ 
Let us recall that the function $u:X\to\R\cup\{+\infty\}$ defined by $u = 0$ on $S$ and $u = +\infty$ on $X \setminus S$ is a supersolution of~\eqref{eqn: G0}. Moreover, since $u$ is larger than any subsolution, Theorem~\ref{thm: T characterization} implies that $T^\infty u=u_0$. 
So, a direct computation gives that, for any $x\in X$
\[ u_0 (x)=\inf \left\{\sum_{i=0}^{k-1} \ell(x_i)d(x_i,x_{i+1}):~ k \in \N, \, \{x_i\}_{i=0}^k\subset X,~x_0=x,~x_k\in [\ell=0] \right\}.\]
Then, it readily follows that $u_0(x) \leq v_{R}(x)$ for all $x \in X$ and $R > 0$. 
Hence, it holds $u_0 \leq v_\infty$.

\medskip 

We show that $v_\infty \leq u_0$. To do this, we check that $v_\infty$ is a subsolution of~\eqref{eqn: G0}. 
Reasoning towards a contradiction, assume that there exist $\epsilon>0$ and $\bar x \in X$ such that $G[v_\infty](\bar x) > \ell(\bar x) + \epsilon$. 
By definition of the global slope, there is $y\in X\setminus \{\bar x\}$ such that
    \begin{equation}\label{Contra.G}
        v_\infty(\bar x) > v_\infty(y) + (\ell(\bar x) + \epsilon) d(\bar x, y).
    \end{equation}
Noticing that $v_R(y) \xrightarrow{R \rightarrow \infty} v_\infty(y)$, we fix $R > 0$ such that 
    \begin{align*}
       v_\infty(y) + \epsilon d(\bar x, y) > v_R(y) \text{ and } R \geq d(\bar x, y).
    \end{align*}
Then, from inequality~\eqref{Contra.G}, we get $        v_R(\bar x)\geq v_\infty(\bar x) > v_R(y) + \ell(\bar x) d(\bar x, y).$
Consequently, $G_R[v_R](\bar x) > \ell(\bar x)$, which contradicts Proposition~\ref{prop: solution semiglobal}.\\

$(ii).$ 
Fix $x \in X$ and $ R > 0$. Let $\epsilon > 0$ and fix a 1--Lipschitz curve $\gamma_{\epsilon}: [0, T] \to  X$ such that
    \begin{equation}\label{V_int_gamma.02}
        V(x) \geq \int_0^{T} \ell(\gamma_\epsilon(t))dt - \epsilon,
    \end{equation}
    $\gamma(0)=x$ and $\gamma(T)\in [\ell=0]$.
Since $\ell$ is continuous, $\ell\circ \gamma_\epsilon$ is Riemann integrable. 
Hence, there exists a partition $P := \{ 0 = t_0 < t_1 < \cdots < t_k = T \}$ of $[0, T]$ such that $|t_{i + 1} - t_i| < R$ for all $i \in \{0, \cdots, k - 1\}$ and 
    \begin{equation}\label{V_int_gamma.03}
        \int_0^{T} \ell( \gamma_{\epsilon} (t))dt \geq \sum_{i=0}^{k-1} \ell(\gamma_{\epsilon}(t_i))|t_{i + 1} - t_i| - \epsilon.
    \end{equation}
Set $x_i = \gamma_{\epsilon}(t_i)$ for $i \in \{0, \cdots, k \}$. Since $\gamma_{\epsilon}$ is 1--Lipschitz, one gets 
    \begin{equation}\label{V_int_gamma.04}
        d(x_{i + 1}, x_i) = d(\gamma_{\epsilon}(t_{i + 1}), \gamma_{\epsilon}(t_i)) \leq |t_{i + 1} - t_i| < R.
    \end{equation}
Combining the inequalities \eqref{V_int_gamma.02}, \eqref{V_int_gamma.03} and \eqref{V_int_gamma.04} and using the definition of $v_R$, we obtain
    \begin{equation}
        V(x) \geq \sum_{i=0}^{k-1}  \ell(x_i) d(x_{i + 1}, x_i) - 2 \epsilon \geq v_R(x) - 2\epsilon.
    \end{equation}
Sending $\epsilon$ to $0$, we get that $V(x) \geq v_R(x)$. 
Since $R>0$ is arbitrary, we conclude
    \begin{equation}
        V(x) \geq \lim_{R \rightarrow 0} v_R(x) = \sup_{R > 0} v_R(x)=v_0(x).
    \end{equation}
    So, $V\geq v_0$.

$(iii).$ Thanks to part $(ii)$, it remains to show that $V\leq v_0$. Let $x\in X$. If $\ell(x)=0$, then $v_R(x)=0$ for any $R>0$. So $v_0(x)= 0=V(x)$.
Assume that $\ell(x)>0$.
Let $\rho>0$ be such that $\textup{dist}(x,[\ell=0])>\rho$. Consider the open set 
\[\Omega_\rho:=\bigcup_{y\in [\ell=0]}B_{\rho}(y).\]
Note that $x\notin \Omega_\rho$.
By compactness of $X$ and continuity of $\ell$, we know

\[L_\rho := \inf\left\{\ell(y):~ y\in X\setminus \Omega_{\rho} \right\} > 0.\]
Fix $R>0$. Let $\{x_i\}_{i=0}^k\subset X$ be such that $x_0=x$, $x_k\in [\ell=0]$, $d(x_i,x_{i+1})<R$ for all $i$, and 
\begin{align}\label{eq: uR approx}
v_R(x)+1\geq \sum_{i=1}^{k-1} \ell(x_i)d(x_i,x_{i+1}).
\end{align}
Define
\begin{align}\label{eq: I}
    I:= \{i\in \{0,...,k\}:~ x_i\in X\setminus \Omega_\rho\}.    
\end{align}
It follows that $0\in I$ and $k\notin I$.
Thanks to inequality~\eqref{eq: uR approx} and the definition of $L_\rho$, we get
\begin{align}\label{eq: bounded length 2}
    \dfrac{V(x)+1}{L_\rho}\geq \sum_{i\in I}d(x_i,x_{i+1}).
\end{align}
Observe that above estimate is independent of $R$.
For any $n\in \N$, let $\{x_i^n\}_{i=1}^{k_n}\subset X$ be such that $x_1^n=0$, $x_{k_n}^n\in [\ell=0]$, $d(x_i^n,x_{i+1}^n)<1/n$ for all $i$ and
\begin{align} \label{eq: approx un}
v_{\frac{1}{n}}(x)+\dfrac{1}{n}\geq \sum_{i=1}^{k_n-1}\ell(x_i^n)d(x_i^n,x_{i+1}^n).\end{align}
Let $I_n\subset \{0,...,k_n\}$ be the set defined as~\eqref{eq: I} but using the sequence $\{x_i^n\}_{i=0}^{k_n}$.
Set $I_n^0$ as the largest interval of integers contained in $I_n$ such that $0\in I_n^0$.
Also, define
\[a_n:=\max I_n^0+1 \in I_n\setminus I_n^0.\]
By definition, $x_{a_n}\in \Omega_\rho$.
By compactness of $X$ and continuity of $\ell$, we have 
\[C :=\max\{\ell(x):x\in \overline{\Omega}_\rho\}<+\infty.\]
Since $[\ell=0]$ is compact, there is $y_n\in [\ell=0]$ such that
\[\textup{dist}(x_{a_n},[\ell=0])=d(x_{a_n},y_n)<\rho.\]
Denote by $\gamma_{a_n}^n$ the geodesic (of speed $1$) joining $x_{a_n}$ and $y_n$.
Then,
\begin{align}\label{eq: geodesic an}
    \int_0^{d(x_{a_n},y_n)}\ell(\gamma_{a_n}^n(t))dt\leq Cd(x_{a_n},y_n)\leq C\rho.
\end{align}

For any $i\in I_n^0$, let $\gamma_i^n:[0,d(x_i^n,x_{i+1}^n)] \rightarrow X$ be a geodesic (of speed $1$) joining $x_i^n$ and $x_{i+1}^n$.
Also, we denote by $\gamma^n: [0, T_n] \to X$ the curve obtained by concatenating the curves $\{\gamma_i^n\}_{i=0}^{a_n}$.
Observe that 
\[\textup{Lip}(\gamma^n)\leq 1,\quad \gamma^n(0)=x \quad ~\text{and } \quad \gamma^n(T_n)=y_n\in[\ell=0],\]
where $T_n = \sum_{i = 0}^{a_{n}- 1} d(x_i^n, x_{i + 1}^n) + d(x_{a_n}, y_n)$. 
Denote by $\omega:[0,+\infty)\to[0,+\infty]$ the modulus of continuity of $\ell$, i.e. $\omega$ is the increasing function defined by
\[\omega(r) :=\sup\left\{|\ell(z_1)-\ell(z_2)|:~z_1,z_2\in X,~d(z_1,z_2)\leq r \right\}.\]
Since $X$ is compact and $\ell$ is continuous, $\omega(r)<\infty$ for all $r\geq 0$ and $\textstyle\lim_{r\to 0}\omega(r)=0$.
So, for any $n\in \N$ and $i\in I_n^0$ we have
\begin{equation}\label{eq: geodesic xn}
    \begin{split}
        & ~\left|\ell(x_i^n)d(x_i^n,x_{i+1}^n)-\int_0^{d(x_i^n,x_{i+1}^n)}\ell(\gamma_i^n(t))dt\right| \\
        \leq & ~  \int_0^{d(x_i^n,x_{i+1}^n)}|\ell(\gamma_i^n(t))-\ell(x_i^n)|dt\leq \omega\left(n^{-1}\right) d(x_i^n,x_{i+1}^n).
    \end{split}
\end{equation}
Finally, putting all together we get that
\begin{align*}
    v_{\frac{1}{n}}(x)+\dfrac{1}{n}+C\rho&\geq \sum_{i\in I_n^0}\ell(x_i^n)d(x_i^n,x_{i+1}^n)+\int_0^{d(x_{a_n},y_n)} \ell(\gamma_{a_n}^n(t))dt\\
    &\geq \int_0^T\ell(\gamma(t))dt+\sum_{i\in I_n^0}\left (\ell(x_i^n)d(x_i^n,x_{i+1}^n)-\int_0^{d(x_i^n,x_{i+1}^n)}\ell(\gamma_i^n(t))dt\right )\\
    &\geq V(x) - \omega(n^{-1})\sum_{i\in I_n^0} d(x_i^n,x_{i+1}^n)\\
    &\geq V(x)- \omega(n^{-1})\dfrac{V(x)+1}{L_\rho},
\end{align*}
where in the first line we use~\eqref{eq: approx un} and~\eqref{eq: geodesic an}, in the third one~\eqref{eq: solution local} and~\eqref{eq: geodesic xn} and in the last line~\eqref{eq: bounded length 2}.
Sending $n$ to $+\infty$, we get
\[v_0(x)+C\rho\geq V(x).\]
Finally, sending $\rho$ to $0$ we obtain $v_0(x)\geq V(x)$. 
The proof is now complete.
\end{proof}
\begin{remark}\label{rem: above}\normalfont
Without significant changes of the above proof, it can be shown that $\textstyle\lim_{R\to 0}v_R=V$ whenever $X$ is a complete length space, $\ell:X\to[0, + \infty)$ is uniformly continuous, $[\ell=0] \neq \emptyset$ and the Hausdorff-Pompeiu distance $D_H([\ell=0],[\ell\leq \alpha] )$ tends to $0$ as $\alpha$ tends to $0^+$. 
\end{remark}
We end this section by showing that Theorem~\ref{thm: local and global}~$(iii)$ does not hold if the function $\ell$ is not assumed to be uniformly continuous. 
We denote again by $v_R$ the function defined by~\eqref{eqn: fn vR}. Also, as in Theorem~\ref{thm: local and global}, we denote by $v_0:X\to\R $ the function 
    \[
        v_0(x) := \lim_{R\to 0} v_R(x) =\sup\left\{v_R(x):~R>0\right\}, \text{ for every } x \in X.        
    \]
\begin{example}\label{thm: counterexample}
    There exist a complete bounded geodesic metric space $X$ and continuous bounded function ${\ell: X\to [0, + \infty)}$, with $[\ell=0] \neq \emptyset$, such that $v_0$ is not a solution of~\eqref{eqn: Ls}.
\end{example}
    Let $X$ be the set defined by
    \[X := \{0\}\cup\{1\}\cup \{(x,n):~x\in (0,1),~n\in \N\}\]
    and let $d:X\times X\to \R$ be the function defined by
    \[d(z_1,z_2)=\begin{cases}
        0 &~\textup{if }z_1=z_2,\\
        
        1 &~\text{if }z_1=0~\text{and } z_2=1,\\
        x &~\text{if }z_1=0~\text{and } z_2=(x,n),\\
        |x-y|&~\text{if }z_1=(x,n)~\text{and } z_2=(y,n),\\
        1-x&~\text{if }z_1=1~\text{and } z_2=(x,n),\\        
        \min \{x+y, 2-x-y\}&~\textup{if }z_1=(x,n),~z_2 = (y, m)~\text{and } n\neq m.\\
    \end{cases}\]
    All the unwritten cases are defined by symmetry. 
    Observe that $(X, d)$ is a metric space.
    For any $n\in \N$, let us denote by $S_n$ the segment
    \[S_n:=\{0\}\cup \{(x,n):x\in(0,1)\} \cup \{1\}.\]
    The metric space $X$ can be regarded as countably many segments joining $0$ and $1$. 
    Let us collect some simple facts of $X$: it is a bounded space with $\textup{diam}(X)=1$, complete and not compact (the sequence $\{(1/2,n) \}_n$ has no accumulation points). 
    Moreover, $X$ is a geodesic space.\\
    Let us fix a continuous function $\ell:X\to\R$ with the following properties: 
    \begin{enumerate}
        \item[(P1)]  $\ell(0)=0$, $\ell(1)=1$ and $\ell((x,n))=x,~ \text{for all }x\in\left(0,\dfrac{1}{3}\right) \cup \left(\dfrac{2}{3},1\right)\,\text{and all } n \in \N.$
        \item[(P2)] $\ell((x,n))\geq x$ for all $x\in (0,1)$ and all $n\in \N$. 
        \item[(P3)] $
    \displaystyle\int_0^1 \ell((x,n)) dx = 1,~\text{for all }n\in \N.$    
    \item[(P4)] The set $
        P_n:=\{z\in S_n:~z=(x,n)~\text{and }\ell(z)=x\}$ is a $\frac{1}{n}$-net of $S_n$.
    \end{enumerate}
    We can also choose $\ell$ bounded, for instance, by $3$. 
    Also, due to (P3) and (P4), $\ell$ cannot be uniformly continuous.
    In what follows, we present a simple consequence of \cite[Theorem~3.6]{DLS_2023}.
\begin{proposition}\label{prop: uniqueness local}
    Let $X$ and $\ell$ be the metric space and function defined above. Then the problem \eqref{eqn: Ls} has a unique non negative solution, namely $V:X\to\R$ defined by~\eqref{eq: solution local}.
\end{proposition}
\begin{proof}
    By Theorem~\ref{thm: solution eikonal}, $V$ is a solution of~\eqref{eq: solution local} associated with $\ell$.
    Let $W$ be any other positive solution. 
    That is, $s[V]=s[W]=\ell$ on $X$ and $V(0)=W(0)=0$.
    Let us recall the notion of asymptotically critical sequence defined in \cite[Definition 3.1]{DLS_2023}: a sequence $\{x_n\}_n\subset X$ it is called $s$-asymptotically critical for $V$ if $\{x_n\}_n$ has no convergent subsequence and 
    \[\sum_{n=1}^\infty s[V](x_n)d(x_n,x_{n+1})< + \infty.\]
    Since $s[V]=s[W]=\ell$, the $s$-asymptotically critical sequences for $V$ coincide with the ones for $W$.
    Let $\{x_n\}_n$ be an $s$-asymptotically critical sequence for $V$.  
    By \cite[Remark 3.2]{DLS_2023}, we have \[\liminf_{n\to\infty}s[V](x_n)=0,\quad \text{and so,}\quad \liminf_{n\to\infty} \ell(x_n)=0.\]
    Hence, by definition of $\ell$, $\{x_n\}_n$ has a subsequence converging to $0$.
    This contradicts the fact that $\{x_n\}_n$ has no accumulation points. 
    Thus, there is no $s$-asymptotically critical sequences for $V$ (and neither for $W$).
    Now, as a direct consequence of \cite[Theorem 3.6]{DLS_2023} we obtain that $V=W$.

    We note that an alternative proof can be done by combining~\cite[Theorem 4.5]{GHN_2015} and~\cite[Theorem~1.2]{LSZ_2021}. 
\end{proof}

Now we can proceed with the proof of Example~\ref{thm: counterexample}.
\begin{proof}[Proof of Example~\ref{thm: counterexample}]
    We show that the metric space $X$ and the function $\ell$ defined above are such that ${v_0:=\lim_{R\to 0} v_R}$ is not a solution of the equation~\eqref{eqn: Ls}.
    Let $R > 0$. 
    Then, for any $n\in \N$ such that $nR > 1$, we have
    \begin{align*} 
        v_R(1)&=\inf \left\{ \sum_{i=0}^{k-1} \ell(z_i)d(z_i,z_{i+1}) : \{z_i\}_i\subset X,~z_0=1,~z_k=0,~d(z_i,z_{i+1})<R \right\}\\
        &\leq \inf \left\{ \sum_{i=0}^{k-1} x_i|x_i-x_{i+1}| : \{(x_i,n)\}_{i=1}^{k-1}\subset P_n, x_0=1,x_k=0,~|x_i-x_{i+1}|<R \right\}.
    \end{align*}
    Since $P_n$ is a $\frac{1}{n}$-net of $S_n$, (P4), and since the function $h(x)=x$ is Riemann integrable in $[0.1]$, we get
    \[v_R(1)\leq \liminf_{n\to\infty} \inf\left\{\sum_{i=0}^{k-1} x_i|x_i-x_{i+1}| :\{(x_i,n)\}_{i=1}^{k-1}\subset P_n, x_0=1,x_k=0,~|x_i-x_{i+1}|<R \right \} \leq \int_0^1xdx=\dfrac{1}{2}. \]
    Therefore, $v_0(1) = \lim_{R\to0} v_R(0)\leq \frac{1}{2}$. 
    On the other hand, by Proposition~\ref{prop: uniqueness local} we know that the only non negative solution of \eqref{eqn: Ls} is the function $V$ defined by~\eqref{eq: solution local}. 
    Due to (P3), we have that
    \[V(x):= \inf \left\{\int_0^{T} \ell(\gamma(t))dt:\textup{Lip}(\gamma)\leq 1, \gamma(0)=1,\gamma(T)=0 \right\} = 1.\]
    Therefore, $v_0\neq V$, and thus, $v_0$ is not a solution of~\eqref{eqn: Ls}.
\end{proof}

\subsection{Integration formula}\label{subsec: integration}

In this subsection, we explore the case in which the lsc function $\ell$ admits the value $+\infty$. 
The main idea is to apply our previous results to the subset $X_f = [\ell< + \infty] \subset X$ and then to consider its lsc envelope. 
As a byproduct of our result, we provide an integration formula for lsc bounded from below functions defined on a complete metric space in terms of their global slopes.
Some further consequences of our integration formula are also established at the end  of the subsection.

\medskip

To fix notation, for a function $f:X \to \R\cup\{+\infty\}$, we denote by $\mathrm{lsc}(f):X\to\R\cup\{+\infty\}$ its lsc envelope, that is, 

\[\mathrm{lsc}(f)(x):= \liminf_{y \to x} f(y), \text{ for all } x \in X.
\]
Observe that if $f$ is lsc on $\mathrm{dom}~f$, then $f = \mathrm{lsc}(f)$ on $\mathrm{dom}~f$. 

In the language of the Global slope equation, we have the following existence result.
\begin{proposition}\label{thm: ell infty}
Let $X$ be a metric space and let $\ell:X\to[0, + \infty]$ be a lsc function satisfying assumption~\eqref{eq: hyp}. 
Then, the equation~\eqref{eqn: G0} admits a lsc solution.
\end{proposition}

\begin{proof}
Since $\ell$ satisfies \eqref{eq: hyp}, $X_f:=[\ell< + \infty]$ is a nonempty set. 
Consider the function ${u_f: X \to \R}$ defined by
\[u_f(x):=\inf \left\{\sum_{n=0}^\infty \ell(x_n)d(x_n,x_{n+n}):~\{x_n\}_n\subset X,~x_0=0, \lim_{n\to\infty}\ell(x_n)=0 \right\},~\text{for all }x\in X.\]
Observe that $u_f(x) < + \infty$ if and only if $x \in X_f$.
Moreover, $u_f \vert_{X_f}$ is the Perron solution of equation~\eqref{eqn: G0} on $X_f$ associated with $\ell \vert_{X_f}$. 
Let $u:X\to \R\cup\{+\infty\}$ be the lsc envelope of $u_f$, i.e. $u=\mathrm{lsc}(u_f)$.\\

\textbf{Claim:} \textit{$u$ is a solution of \eqref{eqn: G0}.} 
Let $x\in X$.
If $x\in X_f$, since $u(x) = \mathrm{lsc}(u_f)(x)$, it readily follows that 
\[G[u](x)=\sup_{y\in X\setminus \{x\}} \dfrac{\max\{u(x)-u(y),0\}}{d(x,y)}=\sup_{y\in X_f\setminus \{x\}} \dfrac{\max\{u(x)-u(y),0\}}{d(x,y)}=G_{X_f}[u_f \vert_{X_f}](x)=\ell(x).\]
If $x\in X\setminus \overline{X_f}$, it follows that $u=+\infty$. Thus, $G[u](x)=+\infty$.
So, we only need to analyse the case when $x\in \overline{X_f} \setminus X_f$.\\

\textbf{Case 1}: $u(x)= +\infty$. 
Then, $G[u](x)=+\infty$.
Since $x\in \overline{X_f} \setminus X_f$, we also have that $\ell(x)=+\infty$.\\

\textbf{Case 2}: $u(x)< + \infty$.
We need to show that $G[u](x)= + \infty$.
Fix $R>0$.
Since $\ell$ is lsc, there is $\delta>0$ such that $\ell(y)>R$ for all $y\in B(x,\delta)$. 
Let $\rho>0$. 
By the definition of $u_f$, for any $y\in B(x,\delta/2)\cap X_f$, there is a sequence $\{x^y_n\}_n \subset X$ such that $x^y_0=y$, $\lim_{n \to \infty} \ell(x^y_n)=0$ and
\[
u_f(y)+\rho\geq \sum_{n=0}^\infty \ell(x^y_n)d(x^y_n,x^y_{n+1}).
\]

Set $N = N(y) :=\min \{n\in\N:~x^y_n\notin B(x,\delta)\}$. 
Then, we have
\[u_f(y)+\rho\geq \sum_{n=0}^{N-1} \ell(x^y_n)d(x^y_n,x^y_{n+1})+u_f(x^y_{N})\geq Rd(y,x^y_{N})+u_f(x^y_{N}),\]
which leads to
\[\dfrac{u_f(y)-u_f(x^y_{N})}{d(y,x^y_{N})}\geq R-\dfrac{\rho}{d(y,x^y_{N})}.\]
Fix $\varepsilon>0$ and let $y\in B(x,\delta/2)$ be such that $|u(x)-u(y)|<\varepsilon$.
Recalling that $u_f =u $ on $X_f$, we get that 
\begin{equation}
    \begin{split}\label{esti.x_yK}
    \dfrac{u(x)-u(x^y_{N})}{d(x,x^y_{N})}&=\dfrac{d(y,x^y_{N})}{d(x,x^y_{N})}\left( \dfrac{u(x)-u(y)+u(y)-u(x^y_{N}) }{d(y,x^y_{N})} \right)\\
    &\geq \dfrac{d(y,x^y_{N})}{d(x,x^y_{N})}\left(
    \dfrac{-\varepsilon}{d(y,x^y_{N})}+R-\dfrac{\rho}{d(y,x^y_{N})}
    \right).
    \end{split}
\end{equation}

Notice that 
    \begin{align*}
        \lim_{y \to  x} \dfrac{d(y,x^y_{N})}{d(x, x^y_{N})} = 1 \quad \text{ and } \quad \liminf_{\substack{ y \in X_f \\ y \to x}} d(y,x^y_{N})\geq \delta.
    \end{align*}
From inequality~\eqref{esti.x_yK}, after taking limit inferior as $y\in X_f$ tends to $x$, we obtain
\[G[u](x)\geq \liminf_{\substack{y \in 
 \mathrm{dom}~\ell \\y  \to x}}\dfrac{u(x)-u(x^y_{N})}{d(x,x^y_{N})} \geq R-\dfrac{\varepsilon+\rho}{\delta}. \] 
Letting $\varepsilon$ tend to $0$ and then $\rho$ tend to $0$, we deduce that $G[u](x)\geq R$. 
Since the above conclusion holds for any $R>0$, we obtain that $G[u](x)= + \infty$.
The proof is complete.
\end{proof}

In the following theorem, we provide an integration formula for proper lsc bounded by below functions defined on a complete metric space via their global slopes. 
For the sake of convenience, for any fixed proper lsc function ${f: X \to \R}$, we denote
 	\begin{equation}\label{eq.integral}
 	\mathcal{I}[f](x) := \inf \left\{ \sum_{n = 0}^{+ \infty} G[f](x_n)d(x_n, x_{n + 1}): \{ x_n \}_n \in \mathcal{K}(x) \right\}, \text{for every } x \in X,
 	\end{equation}
where $\mathcal{K}(x) := \left\{ \{ x_n \}_n \subset X : x_0 = x \text{ and } \lim_{n \rightarrow \infty} G[f](x_n) = 0 \right\}$. Observe that if $f$ is bounded  from below, then $\mathcal{I}[f](x) < + \infty$ iff $G[f](x) < + \infty$.  

\begin{theorem}\label{corol.recovering}
Let $X$ be a complete metric space and $f: X \to \R \cup \{ + \infty \}$ be a proper lsc function with $\inf_X f = 0$. Then, one has
	\begin{align*}
		f = \mathrm{lsc}(\mathcal{I}[f]) \text{ in  } X. 
	\end{align*}
\end{theorem}

\begin{proof}Thanks to Remark~\ref{Ekeland.rmk} ~(ii), we note that $\mathrm{dom}~f \subset \overline{\mathrm{dom}}~G[f]$.  First, we prove that $f \leq \mathrm{lsc}(\mathcal{I}[f])$ in $X$. Denote $D:= \mathrm{dom}~G[f]$. We recall that $\mathcal{I}[f] \big\vert_D$ is the Perron solution of the equation
	\begin{equation}\label{S[f].eqn}
		\begin{dcases}
			G_D[u](x) = G[f](x), \, x \in D, \\
			\inf_D u = 0,
		\end{dcases}
	\end{equation}
where $G_D$ is the global slope operator for functions defined on $D$. 
Notice that $\inf_X f = \inf_D f = 0$.

\medskip

\textbf{Claim}: $G_D[f \big\vert_D](x) = G[f](x)$ for every $x \in D$.

\medskip
For the sake of brevity, we write $G_D[f]$ instead of $G_D[f \big\vert_D]$. 
Reasoning towards a contradiction, assume that there is $\bar x \in D$ such that $G_D[f](\bar x) \neq G[f](\bar x)$. Therefore, $G_D[f](\bar x) < G[f](\bar x)$. 
Using the definition of global slope, there exists $\bar y \in X \setminus D$ such that
	\begin{align*}
		\dfrac{f(\bar x) - f(\bar y)}{d(\bar x, \bar y)} > G_D[f](\bar x),
	\end{align*}
which leads to
	\begin{equation}\label{reco.01}
		f(\bar y) < \liminf_{\substack{y \to \bar y \\ y \in D}} f(y). 
	\end{equation}
This contradicts Remark~\ref{Ekeland.rmk}~(ii), which asserts that there is always a sequence $\{ y_n \}_n \subset D$ such that ${y_n \xrightarrow{n \to \infty} y}$ and $\textstyle\lim_{n \to \infty} f(y_n) \leq f(y)$. This finishes the proof of the claim.

\medskip

So, the function $f$ satisfies
	\begin{equation}\label{f.G_D}
		\begin{dcases}
			G_D[f](x) = G[f](x), \, x \in D, \\
			\inf_D f = 0.
		\end{dcases}	
	\end{equation}	 
Since $\mathcal{I}[f] \big\vert_D$ is the Perron solution of~\eqref{S[f].eqn}, it follows that $f \leq \mathcal{I}[f]$ on $D$. Further, since $f$ is lsc and $f=\mathrm{lsc}(\mathcal{I}[f])= \infty$ on $X\setminus \overline{D}$, we obtain $f \leq \mathrm{lsc}(\mathcal{I}[f])$ in $X$. 

\medskip

It remains to show that $\mathrm{lsc}(\mathcal{I}[f]) \leq f$. Fix $\sigma \in (0, 1)$ and $x \in D$. Applying Remark~\ref{Ekeland.rmk}(i) to the function $f$ at $x$ and $\sigma$, there exists a sequence $\{ z_n \}_n$ starting at $x$ that satisfies
	\begin{equation}\label{eke.sum}
		\lim_{n \to \infty} G[f](z_n) = 0, \quad \sum_{n = 0}^{\infty} G[f](z_n)d(z_n, z_{n + 1}) < + \infty,
	\end{equation}
and
	\begin{equation}\label{eke.differ}
		(1 - \sigma)G[f](z_n)d(z_n, z_{n + 1}) \leq f(z_n) - f(z_{n + 1}) \text{ for every } n \geq 0. 
	\end{equation}
Moreover, up to a subsequence, we can and shall assume that $\{ z_n \}_n$ satisfies~\eqref{eke.sum}--\eqref{eke.differ} and $\{G[f](z_n)\}_n$ is decreasing. 
Indeed, we can construct such a subsequence as we did in the proof of~\Cref{prop.solu_series}.

\medskip

Applying~\Cref{prop.TZ}, one has
$\liminf_{n \to \infty} f(z_n) = \inf_X f = 0.$
By the definition of $\mathcal{I}[f]$, for any $N \in \N$, we directly obtain
	\begin{equation}
		\begin{split}
			\mathcal{I}[f](x) \leq & \sum_{n = 0}^{+ \infty} G[f](z_n)d(z_n, z_{n + 1}) \\
			\leq & ~ \dfrac{1}{1 - \sigma} \sum_{n = 0}^N (f(z_n) - f(z_{n + 1})) + \sum_{n= N + 1}^\infty G[f](z_n)d(z_n, z_{n + 1}) \\
			\leq & ~ \dfrac{1}{1 - \sigma} (f(x) - f(x_N)) + \sum_{n= N + 1}^\infty G[f](z_n)d(z_n, z_{n + 1}).
		\end{split}
	\end{equation}
Taking first the limit inferior as $N$ tends to $+\infty$ and then sending $\sigma$ to $0$, we infer that $\mathcal{I}[f](x) \leq f(x)$. 
Hence, we get $\mathrm{lsc}(\mathcal{I}[f]) \leq f$ in $D$. 
\medskip 

Fix $x \in \mathrm{dom}~f \setminus D$, by Remark~\ref{Ekeland.rmk}(i), there exists a sequence $\{ x_n \}_n \subset D$, convergent to $x$, such that $\textstyle\lim_{n \to \infty} f(x_n) \leq f(x)$. 
Using the fact that $\mathcal{I}[f] \leq f$ in $D$, we obtain
	\begin{align*}
		\mathrm{lsc}(\mathcal{I}[f])(x) = \liminf_{\substack{y \to x \\ y \in D}} \mathcal{I}[f](y) \leq \liminf_{\substack{y \to x \\ y \in D}} f(y) \leq \lim_{n \to \infty} f(x_n) \leq f(x). 
	\end{align*}
So $\mathrm{lsc}(\mathcal{I}[f]) \leq f$ in $\mathrm{dom}~f$, which completes the proof. 
\end{proof}
\begin{remark}\normalfont
    Up to the best of our knowledge,~\Cref{corol.recovering} can be applied to lsc convex functions which are bounded from below to obtain an integration formula using less information than the one given by  Rockafellar in~\cite[Theorem 1]{Rock_66}. 
    Indeed, let $X$ be a Banach space and $f:X\to \R\cup\{+ \infty\}$ be a lsc convex function which is bounded from below.
    Then~\Cref{corol.recovering} implies that $f= \mathrm{lsc}(\mathcal{I}[f])+c$, where $c\in \R$ . 
    On the other hand, if $\partial f$ denotes the convex subdifferential of $f$, since $f$ is convex, it is known that $G[f](x)=\inf\{\|x^*\|:~x^*\in \partial f(x)\}$ when $\partial f(x)\neq \emptyset$ and $+ \infty$ otherwise, see \cite[Proposition 1.4.4]{AGS_2008}.  
\end{remark}

Also, we immediately recover the determination result due to Rockafellar for convex functions which are bounded from below.

\begin{corollary}[Rockafellar's determination result]\label{corol.rocka}
Let $X$ be a Banach space and $f, g: X \to \R \cup \{ + \infty \}$ be convex proper lsc bounded from below functions. 
If $\partial f(x) = \partial g(x)$ for every $x \in X$, then $f = g + c$ in $X$, where $c\in \R$. 
\end{corollary}
\begin{proof}
    Indeed, since $f$ and $g$ are lsc convex functions and $\partial f=\partial g$, we deduce that $G[f]=G[g]$. 
    \Cref{corol.recovering} readily implies that there are $c_1,c_2\in \R$ such that
    \[f+c_1 = \mathrm{lsc}(\mathcal{I}[f])=\mathrm{lsc}(\mathcal{I}[g])=g+c_2.\]
\end{proof}

Even more, we also recover the following result from \cite{IZ_2023}.

\begin{corollary}\label{corol.deter.global}{\normalfont \cite[Theorem 1]{IZ_2023}}
Let $X$ be a complete metric space and $f, g: X \to \R \cup \{ + \infty \}$ be proper lsc functions. Assume that $g$ is bounded from below.
Then, the following assertions hold true:
	\begin{itemize}
		\item[$(i)$] if $G[f](x) \leq G[g](x)$ for every $x \in X$, then one has
			\begin{align*}
				f - \inf_X f \leq g - \inf_X g \text{ in } X;
			\end{align*}
		\item[$(ii)$] $G[f](x) = G[g](x)$ for every $x \in X$ if and only if there exists a constant $c \in \R$ such that $f = g + c$.
	\end{itemize}
\end{corollary}
\begin{proof}
    $(i)$ Thanks to~\Cref{prop.AC-seq} and~\Cref{prop.TZ}, we observe that $f$ is also bounded from below. Now, it is enough to apply~\Cref{corol.recovering}, to realize that $\inf \, (f-\inf f)= \inf \, (g-\inf g)=0$ and that the integration formula~\eqref{eq.integral} is monotone with respect to the global slope.\\
    $(ii)$ It readily follows from $(i)$. 
\end{proof}

We end this paper with a uniqueness result for the equation~\eqref{eqn: G0} which complements Proposition~\ref{thm: ell infty}. 
We stress that the data $\ell:X\to[0, +\infty]$ is only assumed to be lsc and that $\inf_X \ell=0$.
Note that the following result extend
Corollary~\ref{cor: TZDSL} twofold: the solution may not be continuous (but merely lsc) inside the effective domain and the data may not be locally bounded.
Our result reads as follows.

\begin{corollary}\label{corol.unique.improve}
Let $X$  be a complete metric space and let $\ell: X \to [0, + \infty]$ be a lsc function such that $\inf_X \ell = 0$ and~\eqref{eq: hyp} is satisfied. Then, equation~\eqref{eqn: G0} admits a unique proper lsc solution $u: X \to [0, + \infty]$.  
\end{corollary}

\bigskip

\noindent\rule{5cm}{1pt} \smallskip\newline\noindent\textbf{Acknowledgements.} The authors are grateful to Aris Daniilidis for several fruitful discussions and to Olivier Ley for many comments that helped to improve the presentation of the present manuscript.

\vspace{0.5cm}

\noindent Tr\'i Minh L\^E, Sebasti\'an TAPIA-GARC\'IA

\medskip

\noindent Institut f\"{u}r Stochastik und Wirtschaftsmathematik, VADOR E105-04
\newline TU Wien, Wiedner Hauptstra{\ss }e 8, A-1040 Wien\medskip
\newline\noindent E-mail: \{\texttt{minh.le,sebastian.tapia.garcia\}@tuwien.ac.at}
\\
\newline\noindent\texttt{https://sites.google.com/view/tri-minh-le}
\newline\noindent\texttt{https://sites.google.com/view/sebastian-tapia-garcia/}

\medskip

\noindent Research partially supported by the grants: \smallskip\newline 
Austrian Science Fund (FWF P-36344N) (Austria)

\nocite{*}

\begin{thebibliography}{1}

\bibitem{ACCT_2013}\textsc{Y. Achdou, F. Camilli, A. Cutr\`i and N. Tchou}, Hamilton-Jacobi equations constrained on networks, \emph{NoDEA Nonlinear Differential Equations Appl} \textbf{20} (2013), 413--445.

\bibitem{AF_2014} \textsc{L. Ambrosio and J. Feng}, On a class of first order Hamilton-Jacobi equations in metric spaces, \emph{J. Differ. Equations} \textbf{256}(2014), 2194--2245.

\bibitem{AGS_2013} \textsc{L. Ambrosio, N. Gigli and G. Savar\'e}, Heat flow and calculus on metric measure spaces with Ricci curvature bounded below—the compact case, \emph{Analysis and numerics of partial differential equations},  \emph{Springer INdAM Ser., \textbf{4},} 2013, 63--115.

\bibitem{AGS_2008} \textsc{L. Ambrosio, N. Gigli and G. Savar\'e}, \emph{Gradient flows in metric spaces and in the space of probability measures}, \emph{Birkh\"{a}user Verlag, Basel,} 2008.

\bibitem{BC-D_97} \textsc{M. Bardi and I. Capuzzo--Dolcetta,} \emph{Optimal control and viscosity solutions of Hamilton-Jacobi-Bellman equations}, \emph{Birkh\"{a}user Boston, Inc., Boston, MA,} 1997. 



\bibitem{BCD_2018} \textsc{T. Z. Boulmezaoud, P. Cieutat and A. Daniilidis}, Gradient flows, second-order gradient systems and convexity, \emph{SIAM J. Optim.} \textbf{28} (2018), 2049–-2066.


\bibitem{BD_2005} \textsc{A. Briani and A. Davini}, Monge solutions for discontinuous Hamiltonians, \emph{ESAIM Control Optim. Calc. Var.} \textbf{11} (2005), 229–-251.


\bibitem{CCM_2016} \textsc{F. Camilli, R. Capitanelli and C. Marchi}, Eikonal equations on the Sierpinski gasket, \emph{Math. Ann.} \textbf{364} (2016), 1167–-1188. 

\bibitem{CGMQ_2024} \textsc{P. Cannarsa, S. Gaubert, C. Mendico and M. Quincampoix}, Analysis of the vanishing discount limit for optimal control problems in continuous and discrete time, \emph{Math. Control Relat. Fields}, doi:10.3934/mcrf.2024010.

\bibitem{C_2013} \textsc{P. Cardaliaguet}, Notes on Mean Field Games (from P.-L. Lions’ lectures at Coll\`ege de France), 2013.

\bibitem{CLM_2013} \textsc{A. Chambolle, E. Lindgren and R. Monneau}, A H\"older infinity Laplacian, \emph{ESAIM Control Optim. Calc. Var.} \textbf{18} (2012), 799-–835. 

\bibitem{CL_1983} \textsc{M. G. Crandall and P.--L. Lions}, Viscosity solutions of Hamilton-Jacobi equations, \emph{Trans. Amer. Math. Soc.} \textbf{277} (1983), 1--42.

\bibitem{CL_1985} \textsc{M. G. Crandall and P.--L. Lions}, Hamilton-Jacobi equations in infinite dimensions. I. Uniqueness of viscosity solutions, \emph{J. Funct. Anal.} \textbf{62} (1985), 379--396.

\bibitem{CL_1986} \textsc{M. G. Crandall and P--L. Lions}, Hamilton-Jacobi equations in infinite dimensions. II. Existence of viscosity solutions, \emph{J. Func. Anal.} \textbf{65} (1986), 368--405. 

\bibitem{DD_2023} \textsc{A. Daniilidis and D. Drusvyatskiy}, The slope robustly determines convex functions. \emph{Proc. Amer. Math. Soc.} \textbf{151} (2023), 4751--4756.

\bibitem{DLS_2023}\textsc{A. Daniilidis, T. M. Le and D. Salas}, Metric compatibility and determination in complete metric spaces, Math. Z. (to appear).

\bibitem{DMS_2022} \textsc{A. Daniilidis, L. Miclo and D. Salas,} Descent modulus and applications, \emph{J. Funct. Anal.} \textbf{287} (2024) (To appear).

\bibitem{DS_2022} \textsc{A. Daniilidis and D. Salas}, A determination theorem in terms of the metric slope, \emph{Proc. Amer. Math. Soc.} \textbf{150} (2022), 4325--4333.

\bibitem{DST-G_2024} \textsc{A. Daniilidis, D. Salas and S. Tapia-García}, A slope generalization of Attouch theorem, \emph{Math. Program}, https://doi.org/10.1007/s10107-024-02108-w.

\bibitem{DFIZ_2016_invent} \textsc{A. Davini, A. Fathi, R. Iturriaga and M. Zavidovique}, Convergence of the solutions of the discounted Hamilton-Jacobi equation: convergence of the discounted solutions, \emph{Invent. Math.} \textbf{206} (2016), 29--55.

\bibitem{DFIZ_2016} \textsc{A. Davini, A. Fathi, R. Iturriaga and M. Zavidovique}, Convergence of the solutions of the discounted equation:
the discrete case, \emph{Math. Z.} \textbf{284} (2016), 1021--1034.

\bibitem{E-B_2023} \textsc{S. Eriksson-Bique}, Density of Lipschitz functions in energy, \emph{Calc. Var. Partial Differential Equations} \textbf{62} (2023), 23pp.

\bibitem{EGV_2024} \textsc{F. Essebei, G. Giovannardi and S. Verzellesi}, Monge solutions for discontinuous Hamilton-Jacobi equations in Carnot groups, \emph{NoDEA Nonlinear Differential Equations Appl.} \textbf{31} (2024), Paper No. 95.


\bibitem{F_2008} \textsc{A. Fathi}, \emph{Weak KAM Theorem in Lagrangian Dynamics}, Preliminary Version Number 10, June 2008.

\bibitem{JZ_2024} \textsc{O. Jerhaoui and H. Zidani}, Viscosity solutions of Hamilton-Jacobi equations in proper $CAT(0)$ 
spaces, \emph{J. Geom. Anal.} \textbf{34} (2024), Paper No. 4. 



\bibitem{GS_2015} \textsc{W. Gangbo and A. \'Swi\c{e}ch}, Metric viscosity solutions of Hamilton–Jacobi equations depending on local slopes, \emph{Calc. Var.} \textbf{54} (2015), 1183-–1218.

\bibitem{GHN_2015} \textsc{Y. Giga, N. Hamamuki and A. Nakayasu,} Eikonal equations in metric spaces, \emph{Trans. Amer. Math. Soc.} \textbf{367}(2015), 49--66.

\bibitem{IMZ_2013}\textsc{C. Imbert, R. Monneau and H. Zidani}, A Hamilton-Jacobi approach to junction problems and application to traffic flows, \emph{ESAIM Control Optim. Calc. Var.} \textbf{19} (2013), 129--166.

\bibitem{I_1987} \textsc{H. Ishii}, A simple, direct proof of uniqueness for solutions of the Hamilton-Jacobi equations of Eikonal type, \emph{Proc. Amer. Math. Soc.} \textbf{100} (1987), 247--251.

\bibitem{I_1987_Perron} \textsc{H. Ishii}, Perron's method for Hamilton--Jacobi equations, \emph{Duke Math. J.} \textbf{55} (1987), 369--384.

\bibitem{IZ_2023} \textsc{M. Ivanov and N. Zlateva}, Slopes and Moreau-Rockafellar Theorem (preprint),  arXiv:2309.03505v3. 

\bibitem{L1904} \textsc{H. Lebesgue}, Leçons sur l'intégration et la recherche des fonctions primitives, Gauthier-Villars 1904.

\bibitem{LPV_1987} \textsc{P. L. Lions, G. Papanicolaou and S. R. S. Varadhan}, Homogenization of Hamilton-Jacobi equations, unpublished preprint, 1987. 

\bibitem{LSZ_2021} \textsc{Q. Liu, N. Shanmugalingam and X. Zhou,} Equivalence of solutions of eikonal equation in metric spaces, \emph{J. Differential Equations,} \textbf{272}(2021), 979--1014.

\bibitem{LSZ_2024} \textsc{Q. Liu, N. Shanmugalingam and X. Zhou,} Discontinuous eikonal equations in metric measure spaces, \emph{Trans. Amer. Math. Soc.}, https://doi.org/10.1090/tran/9294.


\bibitem{NN_2018} \textsc{A. Nakayasu and T. Namba}, Stability properties and large time behavior of viscosity solutions of Hamilton-Jacobi equations on metric spaces, \emph{Nonlinearity} \textbf{31} (2018),  5147–-5161.

\bibitem{NS_1995} \textsc{R. T. Newcomb II and J. Su},  Eikonal equations with discontinuities, \emph{Differ. Integral Equ.} \textbf{8} (1995),  1947–-1960.

\bibitem{PSV_2021} \textsc{P. P\'erez-Aros, D. Salas and E. Vilches}, Determination of convex functions via subgradients of minimal norm, \emph{Math. Program. Ser. A} {\bf 190} (2021), 561--583.

\bibitem{Rock_66} \textsc{R. T. Rockafellar}, Characterization of the subdifferentials of convex functions, \emph{Pacific J. Math.} \textbf{17} (1966), 497--510.

\bibitem{Rock_70} \textsc{R. T. Rockafellar}, On the maximal monotonicity of subdifferential mappings, \emph{Pacific J. Math.} \textbf{33} (1970), 209--216.


\bibitem{SC_2013} \textsc{D. Schieborn and F. Camilli}, Viscosity solutions of Eikonal equations on topological
networks, \emph{Calc. Var.} \textbf{46} (2013), 671–-686.


\bibitem{TZ_2023} \textsc{L. Thibault and D. Zagrodny}, Determining functions by slopes, \emph{Commun. Contemp. Math.} \textbf{25} (2023), 30pp.


\bibitem{Z_2002} \textsc{C. Z\u{a}linescu}, \emph{Convex analysis in general vector spaces}, \emph{World Scientific Publishing Co., Inc., River Edge, NJ,} 2002.

\bibitem{Z_2010} \textsc{M. Zavidovique}, Existence of $C^{1,1}$ critical subsolutions in discrete weak KAM theory, \emph{J. Mod. Dyn.} \textbf{4} (2010), 693–-714.

\bibitem{Z_2012} \textsc{M. Zavidovique,} Strict sub-solutions and Ma\~n\'e potential in discrete weak KAM theory, \emph{Comment. Math. Helv.} \textbf{87} (2012), 1--39.

\end{thebibliography}
\end{document}